\documentclass{amsart}

\usepackage[british]{babel}
\usepackage[utf8]{inputenc}
\usepackage[T1]{fontenc}
\usepackage[cal=boondox]{mathalfa}
\usepackage{euler}
\usepackage{amsmath}
\usepackage{amssymb}
\usepackage{amsthm}
\usepackage{palatino}
\usepackage{enumitem}
\usepackage{etoolbox}
\usepackage{tikz-cd}
\usepackage{dsfont}
\usepackage{calc}
\usepackage{mathtools}
\usepackage{xstring}
\usepackage[style=alphabetic,backend=biber,opcittracker=true]{biblatex}
\DeclareFieldFormat{eprint:eudml}{%
  \href{http://eudml.org/doc/#1}{EuDML : #1}%
}
\DeclareFieldFormat{eprint:jstor}{%
  \href{http://www.jstor.org/stable/#1}{JSTOR : #1}%
}
\DeclareFieldFormat{eprint:euclid}{%
  \href{http://projecteuclid.org/euclid.cmp/#1}{Project Euclid : #1}%
}
\usepackage{hyperref}
\usepackage[nameinlink]{cleveref}

\Crefname{subsection}{Subsection}{Subsections}
\crefname{subsection}{subsection}{subsections}

\theoremstyle{definition}
\newtheorem{definition}[subsubsection]{\definitionautorefname}
\newcommand{\definitionautorefname}{Definition}

\newtheorem{construct}[subsubsection]{\constructautorefname}
\newcommand{\constructautorefname}{Construction}

\theoremstyle{remark}
\newtheorem{remark}[subsubsection]{\remarkautorefname}
\newcommand{\remarkautorefname}{Remark}
\newtheorem{subremark}[paragraph]{\remarkautorefname}
\newtheorem{supremark}{\remarkautorefname}
\newtheorem{exmp}[subsubsection]{\exmpautorefname}
\newcommand{\exmpautorefname}{Example}
\newtheorem{subexmp}[paragraph]{\exmpautorefname}
\newtheorem{warning}[subsubsection]{\warningautorefname}
\newcommand{\warningautorefname}{Warning}

\theoremstyle{plain}
\newtheorem{thm}[subsubsection]{\thmautorefname}
\newcommand{\thmautorefname}{Theorem}

\newtheorem{supthm}[supremark]{\thmautorefname}
\newtheorem{corlr}[subsubsection]{\corlrautorefname}
\newcommand{\corlrautorefname}{Corollary}

\newcommand{\prorautorefname}{Property}

\newtheorem{pros}[subsubsection]{\prosautorefname}
\newcommand{\prosautorefname}{Proposition}

\newtheorem{lemma}[subsubsection]{\lemmaautorefname}
\newcommand{\lemmaautorefname}{Lemma}

\makeatletter
\newcommand{\addchar}[2]{%
  \@tfor\letter:=#1\do{%
    \letter#2
  }%
}
\makeatother

\newcommand{\id}{\mathds{1}}%
\newcommand{\cat}[1]{\mathfrak{#1}}
\newcommand{\opcat}[1]{\cat{#1}^{\mathrm{op}}}%
\newcommand{\func}[1]{\mathcal{\addchar{#1}{\!}}\,}
\newcommand{\shname}[1]{\StrSplit{#1}{1}{\debnom}{\resnom}\mathit{\mathscr{\debnom}\resnom}}%

\newcommand{\lmtimes}{\mathop{\times}\limits}
\newcommand{\virshf}[1]{\left[\shname{O}_{#1}^{\mathrm{vir}}\right]}
\newcommand{\compac}[1]{\overline{#1}}

\newcommand{\compmodsp}[2][0]{\compac{\mathcal{M}}_{#1,#2}}
\newcommand{\compmodspo}[1][n]{\compmodsp{#1}}
\newcommand{\compmodcurv}[2][0]{\compac{\mathcal{C}}_{#1,#2}}

\newcommand{\modstk}[2][0]{\cat{M}_{#1,#2}}

\newcommand{\stkcurv}[2][0]{\cat{C}_{#1,#2}}
\DeclareMathOperator{\cofib}{cofib} 
\DeclareMathOperator{\map}{hom} 
\DeclareMathOperator{\spec}{Spec} 
 
 \DeclareMathOperator{\ev}{ev}

\DeclareMathOperator{\rmap}{\mathbb{R}\shname{Map}}
\DeclareMathOperator{\pr}{pr} \DeclareMathOperator{\specialis}{sp}
\DeclareMathOperator{\truncat}{\func{t}_{0}}

\DeclarePairedDelimiter{\uly}{\lvert}{\rvert}
\DeclarePairedDelimiter{\restr}{{}}{\rvert}

\title{A categorification of the quantum Lefschetz % hyperplane
  principle}%
\author{David Kern}%
\address{David Kern, IMAG, Université de Montpellier, CNRS,
  Montpellier, France}%
\email{\href{mailto:david.kern@umontpellier.fr}{david.kern@umontpellier.fr}}%
\urladdr{\href{https://dskern.github.io/}{https://dskern.github.io/}}%
\thanks{The author acknowledges funding from the grant of the Agence
  Nationale de la Recherche ``Categorification in Algebraic Geometry''
  ANR-17-CE40-0014}

\begin{document}

\begin{abstract}
  The quantum Lefschetz formula explains how virtual fundamental
  classes (or structure sheaves) of moduli stacks of stable maps
  behave when passing from an ambient target scheme to the zero locus
  of a section. It is only valid under special assumptions (genus $0$,
  regularity of the section and convexity of the bundle). In this
  paper, we give a general statement at the geometric level removing
  these assumptions, using derived geometry. Through a study of the
  structure sheaves of derived zero loci we deduce a categorification
  of the formula in the $\infty$-categories of quasi-coherent
  sheaves. We also prove that Manolache's virtual pullbacks can be
  constructed as derived pullbacks, and use them to recover the
  classical Quantum Lefschetz formula when its hypotheses are
  satisfied.
\end{abstract}

\maketitle{}

\tableofcontents{}

\section{Introduction}
\label{sec:introduction}

\renewcommand{\thesupremark}{\Alph{supremark}}

\subsection{The quantum Lefschetz hyperplane principle}
\label{sec:quant-lefsch-hyperpl}

Any quasi-smooth derived scheme is Zariski-locally presented as the
(derived) zero locus of a section of a vector bundle on some smooth
scheme. The Lefschetz hyperplane theorem then gives a way of
understanding the cohomology of such a zero locus from the data of
that of the ambient scheme and of the vector bundle. The quantum
Lefschetz principle, similarly, gives the quantum cohomology, that is
the Gromov--Witten theory, of the zero locus from that of the ambient
scheme and the Euler class of the vector bundle.

Let $X$ be a smooth projective variety and let $E$ be a vector bundle
on $X$, and consider the abelian cone stack
$\mathbb{R}^{0}\func{p}_{\ast}\ev^{\ast}E$ on
$\compmodsp[g]{n}(X,\beta)$, where
$\ev\colon\compmodcurv[g]{n}(X,\beta)\to X$ is the canonical
evaluation map (corresponding by the isomorphism
$\compmodcurv[g]{n}(X,\beta)\simeq\compmodsp[g]{n+1}(X,\beta)$ to
evaluation at the $(n+1)$th marking) and
$\func{p}\colon\compmodcurv[g]{n}\to\compmodsp[g]{n}$ is the
projection. Let $s$ be a regular section of $E$ and
$i\colon Z\hookrightarrow X$ be its zero locus. An inspection of the
moduli problems (see the proof of~\cref{corlr:main-result}) reveals
that the disjoint union, over all classes $\gamma\in A_{1}Z$ mapped by
$i_{\ast}$ to $\beta$, of the moduli stacks of stable maps to $Z$ of
degree $\gamma$ coincides with the zero locus of the induced section
$\mathbb{R}^{0}\func{p}_{\ast}\ev^{\ast}s$ of
$\mathbb{R}^{0}\func{p}_{\ast}\ev^{\ast}E$. The natural question,
leading to the quantum Lefschetz theorem, is whether this
identification remains true at the ``virtual'' level, which was
conjectured by Cox, Katz and Lee in~\cite[Conjecture
1.1]{cox01:_virtual}. It was indeed proved
in~\cite{kim03:_funct_cox_katz_lee} for Chow homology, and the
statement was lifted in~\cite{joshua10:_rieman} to $G_{0}$-theory,
that under assumptions on $E$ the Gromov--Witten theory of $Z$ is
equivalent to that of $X$ twisted by the Euler class of $E$, in that
the following holds.
\begin{supthm}[{\cite{kim03:_funct_cox_katz_lee,joshua10:_rieman}}]
  \label{thm:state-art-qlefschetz}
  For any $\gamma\in A_{1}Z$ such that $i_{\ast}\gamma=\beta$, let
  $u_{\gamma}\colon\compmodspo(Z,\gamma)
  \hookrightarrow\compmodspo(X,\beta)$ denote the closed immersion.
  Suppose $E$ is \textbf{convex}, that is
  $\mathbb{R}^{1}p_{\ast}(C,\mu^{\ast}E)=0$ for any stable map
  $\mu\colon C\to X$ from a rational (\emph{i.e.} genus-$0$) stable
  curve $C\xrightarrow{p}S$ (so that the cone
  $\mathbb{R}^{0}\func{p}_{\ast}\ev^{\ast}E$ is a vector bundle). Then
  \begin{equation}
    \label{eq:kkp}
    \sum_{i_{\ast}\gamma=\beta}u_{\gamma,\ast}
    \left[\compmodspo(Z,\gamma)\right]^{\mathrm{vir}}
    =\left[\compmodspo(X,\beta)\right]^{\mathrm{vir}}
    \smile\operatorname{c}_{\mathrm{top}}(
    \mathbb{R}^{0}\func{p}_{\ast}\ev^{\ast}E) \in
    A_{\bullet}\left(\compmodspo(X,\beta)\right)\text{,}
  \end{equation}
  and
  \begin{equation}
    \label{eq:joshua}
    \sum_{i_{\ast}\gamma=\beta}u_{\gamma,\ast} \virshf{\compmodspo(Z,\gamma)}
    =\virshf{\compmodspo(X,\beta)} \otimes
    \lambda_{-1}(\mathbb{R}^{0}\func{p}_{\ast}\ev^{\ast}E) \in
    G_{0}\left(\compmodspo(X,\beta)\right)\text{.} 
  \end{equation}
\end{supthm}

It was shown in~\cite{coates12:_lefsc} that the quantum Lefschetz
principle as stated in~\eqref{eq:kkp} can be false when the vector
bundle $E$ is not convex (or as soon as $g$ is greater than $0$).  The
reason for this is that $\mathbb{R}^{0}\func{p}_{\ast}\ev^{\ast}E$ no
longer equals $\mathbb{R}\func{p}_{\ast}\ev^{\ast}E$ and the twisting
Euler class should be corrected by taking into account the term
$\mathbb{R}^{1}\func{p}_{\ast}\ev^{\ast}E$: in other words, one should
use the full derived pushforward and view the induced cone as a
\emph{derived} vector bundle $\mathbb{R}\func{p}_{\ast}\ev^{\ast}E$.
This will require viewing our moduli stacks through the lens of
derived geometry.

In this note, we use this philosophy to undertake the task of
simultaneousley relaxing the hypotheses
on~\cref{thm:state-art-qlefschetz} and lifting it to a categorified
(and a geometric) statement, by which we mean that:
\begin{itemize}
\item we will give a formula at the level of a derived
  $\infty$-category of quasicoherent sheaves,
\item we will not need to fix the genus to $0$,
\item we will not need to assume that $E$ is convex, or in fact a
  classical vector bundle (\emph{i.e.} it can come from any object
  of the $\infty$-category $\cat{Perf}(\shname{O}_{X})$),
\item we will not need to assume that the section is regular, as we
  can allow the target to be any derived scheme (or even a
  $1$-algebraic derived stack) rather than a smooth scheme.
\end{itemize}

We note however that only the categorified form of the formula will
hold in full generality, as the usual convexity (and genus) hypotheses
are still needed to ensure bounded-coherence conditions so as to
decategorify to $G$-theory.

\subsection{Derived moduli stacks and virtual classes}
\label{sec:deriv-moduli-stacks}

In~\cite{mann18:_brane_gromov_witten_k}, the categorification of
Gromov--Witten classes, as a lift from operators between
$G_{0}$-theory groups to dg-functors between dg-categories of
quasicoherent (or coherent, or perfect) $\shname{O}$-modules, was
achieved through the use of derived algebraic geometry. Indeed, this
language allows one to interpret the homological corrections appearing
in classical algebraic geometry as actual geometric objects; in
particular the virtual structure sheaf $\virshf{\compmodspo(X,\beta)}$
was realised as the actual structure sheaf of a derived thickening
$\mathbb{R}\compmodsp[g]{n}(X,\beta)$ of the moduli stack, so that
applying the $(\infty,2)$-functor $\cat{QCoh}$ to the appropriate
correspondences produces the desired lift of Gromov--Witten theory.

The idea of viewing the virtual fundamental class as a shadow of a
higher structure sheaf was introduced
in~\cite{kontsevich95:_enumer_ration_curves_via_torus_action}, and
made more precise first in~\cite{ciocan-fontanine09:_virtual} using
the language of dg-schemes and in~\cite[\S{}3.1]{toen14:_deriv}
\emph{via} derived geometry. The derived moduli stack of stable maps
$\mathbb{R}\compmodsp[g]{n}(X,\beta)$ was constructed
in~\cite{ciocan-fontanine02:_deriv_hilber}
and~\cite{schurg15:_deriv}. Finally,
\cite{mann18:_brane_gromov_witten_k} showed that the virtual structure
sheaf really is given by the structure sheaf of the derived
thickening, or rather its image by the isomorphism expressing that
$G$-theory does not detect thickenings. Hence, in order to
understand~\cref{thm:state-art-qlefschetz} from a completely geometric
point of view, the role of the virtual classes should indeed be played
by derived moduli stacks.

We may now state the main result of this note, which addresses the
question of similarly understanding the virtual statement of the
quantum Lefschetz principle as a derived geometric phenomenon, and of
deducing an expression for the ``virtual structure sheaf'' of
$\coprod_{\gamma}\compmodsp[g]{n}(Z,\gamma)$, understanding along the
way the appearance of the Euler class of the bundle. In the remainder
of this introduction, we shall write
$\mathbb{R}u\colon\coprod_{\gamma}\mathbb{R}\compmodsp[g]{n}(Z,\gamma)
\hookrightarrow\mathbb{R}\compmodsp[g]{n}(X,\beta)$ the canonical
closed immersion (beware that $\mathbb{R}u$ is not a right derived
functor, but simply a morphism of derived stacks which is a thickening
of $u$).

\begin{supthm}[Categorified quantum Lefschetz principle,
  see~\cref{corlr:main-result} and~\cref{pros:zerosec-struct-shf}]
  \label{thm:main-thm}
  Let $X$ be a derived scheme,
  $\shname{E}\in\cat{Perf}^{\geq0}(\shname{O}_{X})$ a co-connective
  perfect module, and $s$ a section of $\mathbb{V}_{X}(\shname{E})$
  with zero locus
  $Z=X\times_{\mathbb{V}_{X}(\shname{E})}^{\mathbb{R}}X$. Write
  $\func{s}\colon\mathbb{E}^{\vee}
  \coloneqq(\mathbb{R}\func{p}_{\ast}\ev^{\ast}\shname{E})^{\vee}
  \to\shname{O}_{\compmodsp[g]{n}(X,\beta)}$ the cosection (of
  modules) corresponding to
  $\mathbb{R}\func{p}_{\ast}\ev^{\ast}s$. There is an equivalence
  \begin{equation}
    \label{eq:quantum-lefschetz-categorified-main}
    (\mathbb{R}u)_{\ast}\shname{O}_{\coprod_{\gamma}\mathbb{R}\compmodsp[g]{n}(Z,\gamma)}
    \simeq\shname{O}_{\mathbb{R}\compmodsp[g]{n}(X,\beta)}
    \otimes\shname{Sym}\left(\cofib(\func{s})\right)/(t-1)
    =\shname{Sym}\left(\cofib(\func{s})\right)/(t-1)%\text{,}
  \end{equation}
  in $\cat{QCoh}\left(\mathbb{R}\compmodsp[g]{n}(X,\beta)\right)$,
  where $\cofib(\func{s})$ denotes the cofibre (or homotopy cokernel)
  of the linear morphism $\func{s}$ and where
  $\shname{Sym}\left(\cofib(\func{s})\right)$ canonically admits an
  $\shname{O}_{\mathbb{R}\compmodsp[g]{n}(X,\beta)}\{t\}$-algebra
  structure.
\end{supthm}

We first notice that, in this categorified statement and unlike in the 
$G$-theoretic one, the Euler class of $\mathbb{E}^{\vee}$ is refined
to one taking into account the section $s$. Nonetheless this is indeed
a categorification of~\cref{thm:state-art-qlefschetz}, as we will
explain in~\cref{corlr:zero-sec-k-thry}
and~\cref{sec:recov-quant-lefsch-kthry}. When $s$ is the zero section,
meaning that $\func{s}$ is the zero morphism, then
$\shname{Sym}(\cofib(\func{s}))
=\shname{Sym}(\mathbb{E}^{\vee}[1])\otimes\shname{O}_{\mathbb{A}^{1}}$,
with
$\shname{Sym}(\mathbb{E}^{\vee}[1])
=\bigwedge^{\bullet}(\mathbb{E}^{\vee})$ so that in that case we do
recover a categorified Euler class. In particular, passing to the
$G_{0}$ groups will indeed provide an identification of the cofibres
of any and all sections, and hence give back~\cref{eq:joshua}; this
is~\cref{corlr:get-back-qlefshtz-gthry}.

The~\namecref{thm:main-thm} will in fact come as a corollary of a
geometric statement, as a translation of the fact that Euler classes
(also known, in the categorified setting, as Koszul complexes)
represent zero loci of sections. Indeed, we will show that the moduli
stack
$\coprod_{\gamma}\mathbb{R}\compmodsp[g]{n}(Z,\gamma)
=\shname{Spec}_{\mathbb{R}\compmodsp[g]{n}(X,\beta)}\left((\mathbb{R}u)_{\ast}
  \shname{O}_{\coprod_{\gamma}\mathbb{R}\compmodsp[g]{n}(Z,\gamma)}\right)$
satisfies the universal property of the zero locus of
$\mathbb{R}\func{p}_{\ast}\ev^{\ast}s$, meaning that
(per~\cref{corlr:main-result}, the \emph{geometric quantum Lefschetz
  principle}) it features in the cartesian square
\begin{equation}
  \label{eq:geom-statmnt-cart-diagr}
  \begin{tikzcd}
    \coprod\limits_{i_{\ast}\gamma=\beta}\mathbb{R}\compmodsp[g]{n}(Z,\gamma)
    \arrow[d,"\mathbb{R}u_{2}"'] \arrow[r,"\mathbb{R}u_{1}"]
    \arrow[dr,phantom,very near start,"\lrcorner"] &
    \mathbb{R}\compmodsp[g]{n}(X,\beta)
    \arrow[d,"\mathbb{R}\func{p}_{\ast}\ev^{\ast}s"] \\
    \mathbb{R}\compmodsp[g]{n}(X,\beta) \arrow[r,"0_{\mathbb{E}}"'] &
    \restr{\mathbb{E}}_{\mathbb{R}\compmodsp[g]{n}(X,\beta)}
  \end{tikzcd}\text{.}
\end{equation}
The formula~\cref{eq:quantum-lefschetz-categorified-main} for its
relative function ring will then be a consequence of the general
result~\cref{pros:zerosec-struct-shf} describing zero loci of sections
of vector bundles.

\begin{supremark}
  While we have written this introduction with the assumption that the
  target $X$ is a scheme for simplicity, the geometric and
  categorified quantum Lefschetz principles are not only valid for
  (derived) schematic targets, but also for orbifold Gromov--Witten
  theory, as foreshadowed by~\cite[Proposition
  5.1]{coates12:_lefsc}. In fact $X$ and $Z$ can be allowed to be
  derived algebraic stacks, and
  $\mathbb{R}\compmodsp[g]{n}(X,\beta)\subset\rmap_{/\modstk[g]{n}^{\mathrm{tw}}}(
  \stkcurv[g]{n}^{\mathrm{tw}},X\times\modstk[g]{n}^{\mathrm{tw}})$
  (where $\stkcurv[g]{n}^{\mathrm{tw}}\to\modstk[g]{n}^{\mathrm{tw}}$
  denotes the universal twisted curve) can be any open substack
  corresponding to a quasimap stability condition, as used for example
  in~\cite{chen19:_virtual} and detailed
  in~\cite[\S{}4.2.1.1]{kern21:_categ_delig}.
\end{supremark}

The original proof of the quantum Lefschetz principle
in~\cite{kim03:_funct_cox_katz_lee} also consisted of applying an
excess intersection formula to a geometric (or homological) statement,
here the fact that the embedding $u$ satisfies the compatibility
condition implying that Gysin pullback along it preserves the virtual
class. The situation was shed light upon
in~\cite{manolache12:_virtual}, where it was shown that, using
relative perfect obstruction theories (POTs), one can construct
\emph{virtual pullbacks}, which always preserve virtual classes. The
embedding $u$ being regular, its own cotangent complex can be used as
a POT to construct a virtual pullback, which evidently coincides with
the Gysin pullback.

Here we will show (in~\cref{sec:funct-inter-theory}) that, much in the
same way as for the virtual classes, the virtual pullbacks may be
understood as coming from derived geometric pullbacks of coherent
sheaves, so that our statement for the embedding of derived moduli
stacks does imply the quantum Lefschetz formula for the virtual
classes (and in fact its standard proof), with the classical convexity
hypotheses now appearing as necessary to make decategorification
possible.

\subsection{Acknowledgements}
\label{sec:acknowledgements}

These results were obtained as part of my PhD thesis written at the
Université d'Angers, and I wish to thank first my advisor Étienne Mann
for suggesting this problem and for numerous discussions. Many thanks
are also due to my co-advisor Cristina Manolache for helping me
understand the construction of virtual pullbacks, to Benjamin Hennion
for explaining to me over which base the proof
of~\cref{pros:perfcone-equiv-rmap} was to take place, and to Massimo
Pippi for explaining the subtleties of coherence and boundedness for
the passage to $K$-theory. I also thank Marc Levine who suggested
that~\cref{rmrk:genal-hmlgy-thries} was possible. I thank the
anonymous referee who suggested many corrections and improvements and
in particular pointed out an error in the statement
of~\cref{pros:zerosec-struct-shf}, and Bertrand Toën who explained a
correction.

\subsection{Notations and conventions}
\label{sec:notat-conv}

We will use freely the language of $(\infty,1)$-categories (referred
to as $\infty$-categories), developed in a model-independent manner
in~\cite{riehl21:_elemen}, and of derived algebraic geometry, as
developed for example in~\cite{toen08:_homot}
and~\cite{lurie19:_spect_algeb_geomet}. The $\infty$-category of
$\infty$-groupoids, also known as that of spaces
in~\cite{lurie09:_higher}, will be denoted $\cat{\infty\text{-}Grpd}$, and
similarly the $\infty$-category of $\infty$-categories is
$\cat{\infty\text{-}Cat}$.

We work over a fixed field $\Bbbk$ of characteristic $0$; hence the
$\infty$-category of $\Bbbk$-module spectra can be modelled as the
localisation of the category of $\Bbbk$-dg-modules along
quasi-isomorphisms, in a way compatible with the monoidal structures
so that connective $\Bbbk$-$\mathcal{E}_{\infty}$-algebras are modelled by
$\Bbbk$-cdgas concentrated in non-positive cohomological degrees. The
$\infty$-category of derived stacks on the big étale $\infty$-site of
$\Bbbk$ will simply be denoted $\cat{dSt}_{\Bbbk}$.

\begin{supremark}
  The geometric and categorified part of our result, that
  is~\cref{sec:zero-loci-vect-bundl}
  (except~\cref{rmrk:excess-inter-lie-envlop} and beyond)
  and~\cref{sec:geom-lefsch-princ} are valid when $\Bbbk$ is any
  $\mathcal{E}_{\infty}$-algebra over the sphere spectrum
  $\mathbb{S}$. However the formation of free spectral algebras does
  not have good finiteness properties, so in order to have bounded
  structure sheaves defining $G$-theory classes we do need to work
  over $\mathbb{Q}$ where the free spectral algebras coincide
  (by~\cite[Proposition 25.2.6.1]{lurie19:_spect_algeb_geomet}) with
  the polynomial construction.
\end{supremark}

We implicitly embed stacks into derived stacks; as such all
construction are derived by default. In particular the symbol $\times$
will refer to the (homotopical) fibre product of derived stacks; the
truncated (\emph{i.e.} strict, or underived) fibre product of
classical stacks will be denoted $\times^{\func{t}}$, that is
$X\times_{Y}^{\func{t}}Z=\truncat(X\times_{Y}Z)$ for $X$, $Y$ and $Z$
classical.

We shall always use cohomological indexing. By a dg-category (over
$\Bbbk$) we will mean a $\Bbbk$-linear stable $\infty$-category. For
any derived stack $X$, one defines its $G_{0}$-theory group $G_{0}(X)$
as the zeroth homotopy group of the $K$-theory spectrum of the
dg-category $\cat{Coh}^{\mathrm{b}}(X)$.

\section{Zero loci of sections of derived vector bundles}
\label{sec:zero-loci-vect-bundl}

\setcounter{subsection}{-1}

\subsection{A spicilege of derived geometry for Gromov--Witten
  theory}
\label{sec:spic-deriv-geom}

The main import of derived geometry in Gromov--Witten theory is to
make the homological objects which appear to correct defaults of
smoothness more natural (and geometric) by incorporating them from the
start as the basic blocks of the theory. Since the complexes
intervening can only be considered up to quasi-isomorphism, this
amounts in essence to replacing the category of $\Bbbk$-modules with
the \emph{derived} category of such as the place in which to define
$\Bbbk$-algebras. Furthermore, in odrer to work properly with
morphisms between derived $\Bbbk$-algebras, it is necessary to take
the derived category not just as a homotopy category, but as a full
$\infty$-categorical localisation.

In this note, working over a base field $\Bbbk$ containing
$\mathbb{Q}$, we will take the view that the the chain complexes of
(classical) $\Bbbk$-modules, seen as objects of the derived
$\infty$-category, are nothing more than models presenting a more
intrinsic notion of ``derived'' (sometimes also called ``animated'')
$\Bbbk$-modules\footnote{Defined more formally as modules over the
  Eilenberg--MacLane spectrum of $\Bbbk$ in the $\infty$-category of
  spectra.}. In other words, rather than constructing
$\infty$-categorical objects from classical ones, we will take the
$\infty$-categorical language as the more primitive one.  As such, for
any (possibly derived) $\Bbbk$-algebra $A$, we will simply call
\textbf{$A$-modules} the objects of the derived $(\infty,1)$-category
of $A$-modules, which in Gromov--Witten theory are usually rather seen
as \emph{complexes} of truncated $A$-modules. Our only exception to
this terminology, for historical reasons as for example
in~\cite{lurie17:_higher_algeb}, will be for the following important
example:

\begin{exmp}[Cotangent complex]
  If $A$ is a truncated $\Bbbk$-algebra and $A\to B$ is an $A$-algebra
  which is truncated as well, its cotangent complex $\mathbb{L}_{B/A}$
  can be seen as enhancing the cotangent module $\Omega^{1}_{B/A}$
  with homological corrections\footnote{In practice, it can be
    constructed, as a left-derived functor of $\Omega^{1}_{-/A}$, by
    taking a semi-free resolution of $B$ and applying
    $\Omega^{1}_{-/A}$ degreewise.} carrying deformation-theoretic
  information; this is the role it plays in the construction of
  virtual classes. Returning now to the case where $A$ is any general
  (\emph{i.e.} derived) $\Bbbk$-algebra, the \textbf{cotangent
    complex} of an $A$-algebra $A\to B$, denoted $\mathbb{L}_{B/A}$,
  can be characterised as representing ($\infty$-categorical)
  $A$-derivations from $B$, so now plays in higher algebra the exact
  same role that the cotangent module plays in classical algebra.
\end{exmp}

The ideas sketched above provide the notion of affine derived
$\Bbbk$-schemes, as the objects of the opposite $\infty$-category to
that of derived $\Bbbk$-algebras. Since our interest is in enumerative
geometry, we shall use the definition of derived $\Bbbk$-stacks as
moduli problems, \emph{i.e.} given by their $\infty$-functors of
points, $\infty$-functors
$\opcat{Aff}_{\Bbbk}=\cat{Alg}_{\Bbbk}\to\cat{\infty\text{-}Grpd}$
satisfying descent conditions for the étale topology on
$\cat{Aff}_{\Bbbk}$.

\begin{exmp}[Quasicoherent modules]
  The assignment to $\spec{A}\in\opcat{Aff}_{\Bbbk}$ of (the maximal
  $\infty$-groupoid of) the $\infty$-category $\cat{QCoh}(\spec{A})$
  of $A$-modules defines by~\cite[Theorem 1.3.7.2]{toen08:_homot} a
  derived stack (of $\infty$-categories) denoted $\cat{QCoh}$, whose
  groupoidal core is viewed as the classifying stack for quasicoherent
  modules. Then, for any derived stack $X$, the $\infty$-category of
  quasicoherent sheaves on $X$ is
  \begin{equation}
    \label{eq:def-qcoh-stck}
    \cat{QCoh}(X)=\hom(X,\cat{QCoh})
    \simeq\varprojlim_{\spec{A}\to X}\cat{QCoh}(\spec{A})\text{.}
  \end{equation}
  That is, a quasicoherent $\shname{O}_{X}$-module $\shname{M}$ is
  given by an $A$-module $\shname{M}_{x}$ for every
  $x\colon\spec{A}\to X$, and base-change isomorphisms
  $f^{\ast}\shname{M}_{x^{\prime}}\xrightarrow{\simeq}\shname{M}_{x}$
  for every morphism $f\colon\spec{A}\to\spec{A^{\prime}}$ of
  $X$-schemes (along with higher compatibilities).
\end{exmp}

Every derived stack $X$ has its \textbf{truncation} $\truncat{X}$, a
classical (higher) stack obtained by restricting the functor of points
$X$ along the inclusion of truncated (or classical) algebras in all
(derived) algebras. The truncation $\infty$-functor
$\truncat\colon\cat{dSt}_{\Bbbk}\to\cat{St}_{\Bbbk}$ is right-adjoint
to an $\infty$-functor
$\func{i}\colon\cat{St}_{\Bbbk}\to\cat{dSt}_{\Bbbk}$ which is fully
faithful (providing an embedding of classical higher stacks into
derived stacks, by viewing them as trivially derived) and will be
omitted from notation.  This is in keeping with our principle of
implicitly embedding stacks into derived stacks; as such all
construction are derived by default. In particular the symbol $\times$
will refer to the fibre product of derived stacks (given on affines by
the ``derived'' tensor product of algebras); the truncated
(\emph{i.e.} strict, or underived) fibre product of classical stacks
will be denoted $\times^{\func{t}}$, that is
$X\times_{Y}^{\func{t}}Z=\truncat(X\times_{Y}Z)$ for $X$, $Y$ and $Z$
classical.

The counit of the adjunction $\func{i}\dashv\truncat$ will be denoted
$\jmath$; its components
$\jmath_{X}\colon\truncat{X}\hookrightarrow X$ are closed immersions
and will play an important role in the construction of virtual
pullbacks in~\cref{sec:defin-from-deriv}.

In~\cref{sec:revi-deriv-moduli}, we will recall in more details the
relevance of derived grometry, and in particular the role played by
the cotangent complex, in Gromov--Witten theory.

\subsection{Vector bundles in derived geometry}
\label{sec:vect-bundl-deriv}

\begin{definition}[Total space of a quasicoherent module]
  \label{def:vect-bundl}
  Let $X$ be a derived Artin $\Bbbk$-stack. For any quasicoherent
  $\shname{O}_{X}$-module $\shname{M}$, the linear derived stack
  $\mathbb{V}_{X}(\shname{M})$ is described by the $\infty$-functor of
  points mapping an $X$-derived stack $\phi\colon T\to X$ to the
  $\infty$-groupoid
  \begin{equation}
    \label{eq:def-vb-value}
    \map_{\cat{QCoh}(T)}(\shname{O}_{T},\phi^{\ast}\shname{M})\text{.}
  \end{equation}

  We call \textbf{abelian cone} over $X$ any $X$-stack equivalent to
  the total space $\mathbb{V}_{X}(\shname{M})$ of a quasicoherent
  $\shname{O}_{X}$-module $\shname{M}$. We shall say that
  $\mathbb{V}_{X}(\shname{M})$ is a \textbf{perfect} cone if
  $\shname{M}$ is perfect (equivalently, dualisable), and a
  \textbf{vector bundle} if $\shname{M}$ is locally free of finite
  rank (as defined in~\cite[Notation
  2.9.3.1]{lurie19:_spect_algeb_geomet}).
\end{definition}

\begin{remark}
  \label{remark:loc-free-vbund-amplitude}
  If $\shname{M}$ is a locally free $\shname{O}_{X}$-module,
  by~\cite[Proposition 2.9.2.3]{lurie19:_spect_algeb_geomet} we may
  take a Zariski open cover $\coprod_{i}U_{i}\to X$ with
  $\restr{\shname{M}}_{U_{i}}$ free of rank $r_{i}$. We deduce from
  this (or from~\cite[Remark 7.2.4.22]{lurie17:_higher_algeb}
  and~\cite[Remark 2.9.1.2]{lurie19:_spect_algeb_geomet}) that any
  locally free module has Tor-amplitude concentrated in degree $0$,
  and it will follow from~\cref{pros:cotgt-cplx-vect-bndl} that any
  vector bundle is smooth over its base.
\end{remark}

\begin{remark}
  \label{rmrk:perfect-cone-specsym}
  If $\shname{M}$ is dualisable, with dual $\shname{M}^{\vee}$, then
  as pullbacks commute with taking duals we have for any
  $\phi\colon T\to X$
  \begin{equation}
    \label{eq:perf-cone-specsym}
    \begin{split}
      \mathbb{V}_{X}(\shname{M})(\phi)
      &=\map_{\cat{QCoh}(T)}(\phi^{\ast}\shname{M}^{\vee},\shname{O}_{T})\\
      &=\map_{\cat{Alg}(\shname{O}_{X})}(
      \shname{Sym}_{\shname{O}_{X}}(\shname{M}^{\vee})
      ,\phi_{\ast}\shname{O}_{T})
      =\shname{Spec}_{X}^{\text{nc}}(\shname{Sym}_{\shname{O}_{X}}(
      \shname{M}^{\vee}))(\phi)
    \end{split}
  \end{equation}
  where $\shname{Spec}_{X}^{\text{nc}}$ denotes the non-connective
  relative spectrum $\infty$-functor. Hence the restriction of
  $\mathbb{V}_{X}$ to $\cat{Perf}(\shname{O}_{X})$ is naturally
  equivalent to the composite
  $\shname{Spec}_{X}^{\text{nc}}
  \circ\shname{Sym}_{\shname{O}_{X}}\circ(-)^{\vee}$. In particular,
  if $\shname{M}$ is a connective module then
  $\mathbb{V}_{X}(\shname{M})$ is a relatively coaffine stack, while
  if $\shname{M}$ is co-connective, so that
  $\shname{Sym}_{\shname{O}_{X}}(\shname{M}^{\vee})$ is a connective
  algebra, $\mathbb{V}_{X}(\shname{M})$ is an affine derived
  $X$-scheme.

  Note however that the $\infty$-functor
  $\shname{Spec}_{X}^{\text{nc}}$ only becomes fully faithful when
  restricted to % either
  connective $\shname{O}_{X}$-algebras (as this restriction is
  equivalent to the Yoneda embedding thereof) but not when acting on
  general $\shname{O}_{X}$-algebras in degrees of arbitrary positivity
  (see for example~\cite{monier21} for a counterexample, as well as
  details on the full-faithfulness of the $\infty$-functor
  $\mathbb{V}_{X}(-^{\vee})$).
\end{remark}

\begin{warning}[Terminology]
  \label{warn:termin-vb-dual}
  Note that our convention for derived perfect cones is dual to that
  used in (among others) \cite{toen14:_deriv} (and dating back to
  EGA2), which defines the total space of a quasicoherent
  $\shname{O}_{X}$-module $\shname{M}$ as the $X$-stack whose sheaf of
  sections is $\shname{M}^{\vee}$, \emph{i.e.} what we denote
  $\mathbb{V}_{X}(\shname{M}^{\vee})$.
\end{warning}

\begin{exmp}
  \label{exmp:derived-cones}
  \begin{enumerate}[label=\roman*.]
  \item If $X$ is a classical Deligne--Mumford stack and $\shname{M}$
    is of perfect amplitude in $[-1,0]$, the truncation
    $\truncat(\mathbb{V}_{X}(\shname{M}[1]^{\vee}))$ is the abelian cone
    Picard stack $\shname{H}^{1}/\shname{H}^{0}(\shname{M}^{\vee})$
    of~\cite[Proposition 2.4]{behrend97}.
  \item By~\cite[Proposition 1.4.1.6]{toen08:_homot},
    $\mathbb{V}_{X}(\mathbb{T}_{X})=TX\simeq\rmap(\Bbbk[\varepsilon],X)$
    is the tangent bundle stack of $X$. More generally, using
    $\Bbbk[\varepsilon_{n}]$ where $\varepsilon_{n}$ is of
    cohomological degree $-n$ (so of homotopical degree $n$) we have
    the shifted tangent bundle
    $T[-n]X\simeq\mathbb{V}_{X}(\mathbb{T}_{X}[-n])$. Dually, one also
    defines the shifted cotangent stack
    $T^{\vee}[n]X=\mathbb{V}_{X}(\mathbb{L}_{X}[n])$.
  \end{enumerate}
\end{exmp}

\begin{lemma}[{\cite[Sub-lemma 3.9]{toen07:_modul}},{\cite[Theorem
    5.2]{antieau14:_brauer}\footnote{The grading convention used
      in~\cite{antieau14:_brauer} is homotopical, in opposition to our
      cohomological convention.}}]
  \label{lem:ab-cones-geometricity}
  Suppose $\shname{M}$ is of perfect Tor-amplitude contained in
  $[a,b]$ (where $a,b\in\mathbb{Z}$). Then the derived stack
  $\mathbb{V}_{X}(\shname{M})$ is $(-a)$-geometric and strongly of
  finite presentation.
\end{lemma}

\begin{construct}
  \label{constr:abcone-fctorl-pushfrwrd}
  For any derived stack $X$, the $\infty$-functor $\mathbb{V}_{X}$
  gives a link between two functorial (in $X$) constructions. On the
  one hand we have the $\infty$-functor
  $(-)_{\mathrm{ét}}\colon\cat{dSt}_{\Bbbk}\to\cat{\infty\text{-}Cat}$
  mapping a derived $\Bbbk$-stack $X$ to its étale $\infty$-topos
  $X_{\mathrm{ét}}$ and a map of derived stacks $f\colon X\to Y$ to
  the direct image $f_{\ast}$ of the induced geometric morphism,
  mapping a sheaf $\shname{F}$ on $\cat{dSt}_{\Bbbk,/X}$ to the sheaf
  $f_{\ast}\shname{F}\colon(U\to Y) \mapsto\shname{F}(U\times_{Y}X\to
  X)$.

  On the other hand, we have the $\infty$-functor $\cat{QCoh}(-)$
  mapping a derived $\Bbbk$-stack $X$ to the underlying
  $\infty$-category of the dg-category $\cat{QCoh}(X)$, and a map
  $f\colon X\to Y$ to $\shname{M}\mapsto f_{\ast}\shname{M}$ (where
  the direct image sheaf is considered an $\shname{O}_{Y}$-module
  through $f^{\sharp}\colon\shname{O}_{Y}\to
  f_{\ast}\shname{O}_{X}$). Then for any $\shname{M}\in\cat{QCoh}(X)$,
  we obtain the functor of points of its total space,
  $\mathbb{V}_{X}(\shname{M})$, which is an étale sheaf on
  $\cat{dSt}_{\Bbbk,/X}$.
\end{construct}

\begin{lemma}
  \label{lemma:ab-cone-nat}
  Let $\cat{dSt}_{\Bbbk}^{(\mathrm{f.coh.d.})}$ denote the wide and
  $2$-full sub-$\infty$-category whose $1$-arrows are the morphisms of
  finite cohomological dimension (see~\cite[Definition A.1.4, Lemma
  A.1.6]{halpern-leistner14:_mappin}). % QCA maps (whose
  % fibres are quasi-compact, with affine automorphism groups of
  % geometric points, and with classical inertia stacks of finite
  % presentation over their truncations, see~\cite[Definition
  % 1.1.8]{drinfeld13:_some_finit_quest_algeb_stack}).
  The
  $\infty$-functors
  $\mathbb{V}_{X}\colon\cat{QCoh}(X)\to X_{\mathrm{ét}}$ assemble into
  a natural transformation
  $\mathbb{V}\colon\cat{QCoh}(-) \Rightarrow(-)_{\mathrm{ét}}$ of
  $\infty$-functors
  $\cat{dSt}_{\Bbbk}^{(\mathrm{f.coh.d.})}\to\cat{\infty\text{-}Cat}$.
\end{lemma}

\begin{proof}
  We must construct, for any $f\colon X\to Y$ and any
  $\shname{M}\in\cat{QCoh}(X)$, an equivalence
  $f_{\ast}(\mathbb{V}_{X}(\shname{M}))
  =\mathbb{V}_{Y}(f_{\ast}\shname{M})$. For any $\phi\colon U\to Y$,
  the base change along $f$ will take place in the cartesian square
  \begin{equation}
    \label{eq:basechg-sqr-vb-nat}
    \begin{tikzcd}
      X\lmtimes_{Y}U \arrow[d,"f\lmtimes_{Y}U"']
      \arrow[r,"X\lmtimes_{Y}\phi"] \arrow[dr,phantom,very near
      start,"\lrcorner"] & X \arrow[d,"f"] \\
      U \arrow[r,"\phi"'] & Y
    \end{tikzcd}\text{.}
  \end{equation}

  Then we have
  $\mathbb{V}_{Y}(f_{\ast}\shname{M})(U)
  =\map_{\cat{QCoh}(U)}(\shname{O}_{U},
  \phi^{\ast}f_{\ast}\shname{M})$ while
  \begin{equation}
    \label{eq:basechg-vb-nat-comput}
    \begin{split}
      f_{\ast}(\mathbb{V}_{X}(\shname{M}))(U)
      &=\map_{\cat{QCoh}(X\times_{Y}U)}(
      \shname{O}_{X\times_{Y}U},(X\times_{Y}\phi)^{\ast}\shname{M})\\
      &\simeq\map_{\cat{QCoh}(X\times_{Y}U)}(
      (f\times_{Y}U)^{\ast}\shname{O}_{U},
      (X\times_{Y}\phi)^{\ast}\shname{M})\\
      &\simeq\map_{\cat{QCoh}(U)}(\shname{O}_{U},
      (f\times_{Y}U)_{\ast}(X\times_{Y}\phi)^{\ast}\shname{M})\text{.}
    \end{split}
  \end{equation}
  By the base-change property of~\cite[Proposition A.1.5
  (3)]{halpern-leistner14:_mappin} % of~\cite[Corollary 1.4.5
  % (i)]{drinfeld13:_some_finit_quest_algeb_stack} (since $f$ is QCA)
  the two coincide.

  Since the isomorphisms appearing in~\cref{eq:basechg-vb-nat-comput}
  and the base-change map are defined from adjunctions, they come
  equipped with functoriality property which furnish the higher
  naturality coherences.
\end{proof}

\begin{remark}
  \label{rmrk:qsm-preserves-perf}
  By~\cite[Theorem 2.1]{toen12:_proper} , if $f\colon X\to Y$ is
  quasi-smooth and proper then $f_{\ast}$ sends perfect
  $\shname{O}_{X}$-modules to perfect $\shname{O}_{Y}$-modules.
\end{remark}

Finally, we shall use the following well-known description of the
cotangent complex of a perfect cone.

\begin{pros}[{\cite[Theorem 5.2]{antieau14:_brauer}}]
  \label{pros:cotgt-cplx-vect-bndl}
  Let $\shname{M}$ be a perfect $\shname{O}_{X}$-module, and write
  $\pi\colon\mathbb{V}_{X}(\shname{M})\to X$ the structure
  morphism. Then
  $\mathbb{L}_{\pi\colon\mathbb{V}_{X}(\shname{M})/X}
  \simeq\pi^{\ast}\shname{M}^{\vee}$.
\end{pros}

\begin{proof}
  The equivalence is established fibrewise in~\cite[Proposition
  7.4.3.14]{lurie17:_higher_algeb}.
\end{proof}

\subsection{Excess intersection formula}
\label{sec:excess-inters-form}

In this~\namecref{sec:excess-inters-form}, we work with a derived
stack $M$ and the closed embedding $u\colon T\hookrightarrow M$ of
derived stacks defined as the zero locus of a section
$s=\shname{Spec}_{M}s^{\sharp}$ of a (relatively affine) perfect cone
$\shname{Spec}_{M}\shname{Sym}_{\shname{O}_{M}}(\shname{F}^{\vee})$ on
$M$: we fix a co-connective (for the relative affineness) perfect
$\shname{O}_{M}$-module $\shname{F}$ and a morphism of
$\shname{O}_{M}$-algebras
$s^{\sharp}\colon\shname{Sym}_{\shname{O}_{M}}(\shname{F}^{\vee})
\to\shname{O}_{M}$, corresponding (by the left-adjoint property of
$\shname{Sym}_{\shname{O}_{M}}$) to the cosection
$\widetilde{s}\colon\shname{F}^{\vee}\to\shname{O}_{M}$ of the module
$\shname{F}^{\vee}$.

\begin{remark}[Notation, derived versus spectral symmetric powers]
  In spectral algebraic geometry, over an $\mathcal{E}_{\infty}$-ring
  spectrum $\shname{O}$, the construction of polynomial algebras,
  usually denoted $\shname{O}[t_{1},\dots,t_{m}]$, differs from that
  of free symmetric algebras, denoted
  $\shname{O}\{t_{1},\dots,t_{m}\}
  =\shname{Sym}_{\shname{O}}(\shname{O}^{\oplus m})$. Working as we do
  in characteristic zero, the difference between the two vanishes;
  however, as we wish to emphasise that the left-adjoint property of
  the symmetric algebra $\infty$-functor is the one that matters for
  us, making the main result of this~\namecref{sec:excess-inters-form}
  valid over not just over our base $\Bbbk$ but over a general ring
  spectrum, we shall use the spectral notation. In particular, the
  affine line over $M$ is
  $\mathbb{A}^{1}_{M} =\shname{Spec}_{M}(\shname{O}_{M}\{t\})
  =\shname{Spec}_{M}\shname{Sym}_{\shname{O}_{M}}\shname{O}_{M}^{\oplus1}$.
\end{remark}

\begin{pros}
  \label{pros:zerosec-struct-shf}
  The derived $M$-stack $T$ may be recovered as the fibre
  \begin{equation}
    \label{eq:zerosec-stck-formula}
    T\simeq\mathbb{V}_{M}(\cofib(\widetilde{s})^{\vee})
    \lmtimes_{\mathbb{A}^{1}_{M}}\{1\}_{M}
  \end{equation}
  for a certain structure of stack over
  $\mathbb{A}^{1}_{M} =\shname{Spec}_{M}(\shname{O}_{M}\{t\})$ on
  $\mathbb{V}_{M}(\cofib(\widetilde{s})^{\vee})$, that is $T$ is the
  relative spectrum of the quotient $\shname{O}_{M}$-algebra
  \begin{equation}
    \label{eq:zerosec-struct-shf-formula}
    u_{\ast}\shname{O}_{T}
    =\shname{Sym}_{\shname{O}_{M}}\bigl(
    \cofib(\widetilde{s})\bigr)/(t-1)\text{,}
  \end{equation}
  where the structure map
  $\epsilon_{\shname{O}_{M}}\colon\shname{O}_{M}\{t\}\to\shname{O}_{M}$
  is the quotient arrow
  $\shname{O}_{M}\{t\}\to\shname{O}_{M}\{t\}/(t-1)\simeq\shname{O}_{M}$
  mapping $t$ to $1$ (\emph{i.e.} corresponding to the identity
  morphism of $\shname{O}_{M}$-modules
  $\id_{\shname{O}_{M}}\colon\shname{O}_{M}\to\shname{O}_{M}$).
  
  More generally, the monad $u_{\ast}u^{\ast}$ on $\cat{QCoh}(M)$
  identifies with tensoring by the algebra
  $\shname{Sym}_{\shname{O}_{M}}(\cofib(\widetilde{s}))/(t-1)$.
\end{pros}

\begin{proof}
  From the canonical fibre sequence
  $\shname{F}^{\vee}\xrightarrow{\widetilde{s}}\shname{O}_{M}
  \to\cofib(\widetilde{s})$ we obtain, by application of the
  $(\infty,1)$-functor $\shname{Sym}_{\shname{O}_{M}}$, an
  $\shname{O}_{M}\{t\}$-algebra structure
  $\shname{O}_{M}\{t\}\coloneqq\shname{Sym}_{\shname{O}_{M}}(\shname{O}_{M})
  \to\shname{Sym}_{\shname{O}_{M}}(\cofib(\widetilde{s}))$.  As
  $\shname{Sym}_{\shname{O}_{M}}$ is a left-adjoint it preserves
  colimits (by~\cite[Theorem 2.4.2]{riehl21:_elemen}) whence the
  latter term, image by $\shname{Sym}_{\shname{O}_{M}}$ of the
  $\shname{O}_{M}$-module
  $0\oplus_{\shname{F}^{\vee}}\shname{O}_{M}^{\oplus1}
  \eqqcolon\cofib(\widetilde{s})$,
  is the pushout of algebras (so by~\cite[Proposition
  3.2.4.7]{lurie17:_higher_algeb} the tensor product)
  $\shname{O}_{M}\otimes_{\shname{Sym}_{\shname{O}_{M}}(\shname{F}^{\vee})}
  \shname{O}_{M}\{t\}$.

  By definition, the algebra
  $\shname{Sym}_{\shname{O}_{M}}(\cofib(\widetilde{s}))
  \otimes_{\shname{O}_{M}\{t\}}\shname{O}_{M}$ under consideration
  fits in the left pushout square in the diagram
  \begin{equation}
    \label{eq:16}
    \begin{tikzcd}
      \shname{Sym}_{\shname{O}_{M}}(\cofib(\widetilde{s}))/(t-1)
      \arrow[from=d] \arrow[from=r] \arrow[from=dr,phantom,very near
      end,"\lrcorner"] &
      \shname{Sym}_{\shname{O}_{M}}(\cofib(\widetilde{s}))
      \arrow[from=d] \arrow[from=r] \arrow[from=dr,phantom,very near
      end,"\lrcorner"] &
      \shname{O}_{M} \arrow[from=d] \\
      \shname{O}_{M} \arrow[from=r,"\epsilon_{\shname{O}_{M}}"] &
      \shname{O}_{M}\{t\} \arrow[from=r,"\shname{Sym}(\tilde{s})"] &
      \shname{Sym}_{\shname{O}_{M}}(\shname{F}^{\vee})\text{.}
    \end{tikzcd}
  \end{equation}
  From the previous discussion the right square is also cocartesian,
  so that the bigger diagram is also a pushout square.  We now observe
  that the lower composite identifies with $s^{\sharp}$ (since the map
  $\epsilon_{\shname{O}_{M}}$ is the counit of the adjunction
  $\shname{Sym}_{\shname{O}_{M}}\dashv\func{frgt}$), so that the big
  pushout square computes the function $\shname{O}_{M}$-algebra of the
  zero locus of $s$.

  Finally, both $u_{\ast}$ and $u^{\ast}$ are left-adjoints, so by the
  homotopical Eilenberg--Watts theorem
  of~\cite{hovey15:_brown_eilen_watts} (see also \cite[Chapter 4,
  Corollary 3.3.5]{gaitsgory17}) their composite $u_{\ast}u^{\ast}$ is
  equivalent to tensoring by $u_{\ast}u^{\ast}\shname{O}_{M}$. This
  can also be seen as a projection formula (proved for example
  in~\cite[Remark 3.4.2.6]{lurie19:_spect_algeb_geomet}) between
  $u_{\ast}(u^{\ast}-\otimes\shname{O}_{T})$ and
  $-\otimes u_{\ast}\shname{O}_{T}$, or indeed, more tautologically,
  as the definition of the direct and inverse image functors from the
  point of view of derived stacks as ringed $\infty$-topoi (from which
  the further identification of the monad structures follows readily).
\end{proof}

\begin{subremark}[Geometric interpretation]    
  Let
  $\overline{s}\colon\mathbb{A}^{1}_{M}\to\mathbb{V}_{M}(\shname{F})$
  be the linearisation of $s$, obtained as the image of
  $\widetilde{s}$ by $\mathbb{V}_{M}$. The zero locus of
  $\overline{s}$ is
  $\restr{\mathbb{A}^{1}_{M}}_{T}\cup\mathbb{A}^{0}_{M\setminus T}$,
  so taking the fibre at any non-zero element $\lambda$ of
  $\mathbb{A}^{1}_{M}$ recovers
  $T\times\{\lambda\}\cup\emptyset\simeq T$.
\end{subremark}

\begin{subexmp}[Koszul complexes]
  \label{rmrk:koszul-cplx}
  Suppose $\shname{F}$ is locally free. Then, passing to a Zariski
  open cover $\coprod U_{i}\to M$, we may assume as
  in~\cref{remark:loc-free-vbund-amplitude} that
  $\restr{\shname{F}}_{U_{i}}$ is free of rank $r_{i}$. Write
  $\restr{\widetilde{s}}_{U_{i}}=(s_{\ell})_{1\leq\ell\leq r}$ in
  coordinates. Then we recover the Koszul complex
  $\bigotimes_{\ell=1}^{r}\cofib(s_{\ell})$, as studied for instance
  in~\cite[\S{}2.3.1]{khan19:_virtual_cartier}
  or~\cite{vezzosi11:_deriv_i_basic}.
\end{subexmp}

Recall that the exterior algebra of the quasicoherent
$\shname{O}_{M}$-module $\shname{F}^{\vee}$ is
$\bigwedge^{\bullet}\shname{F}^{\vee}
\coloneqq\shname{Sym}_{\shname{O}_{M}}(\shname{F}^{\vee}[1])
=\bigoplus_{n\geq0}(\bigwedge^{n}\shname{F}^{\vee})[n]$.

\begin{corlr}[Excess intersection formula]
  \label{corlr:excess-intersct}
  For any quasicoherent $\shname{O}_{T}$-module $\shname{M}$ that is
  the restriction (along $u^{\ast}$) of an $\shname{O}_{M}$-module,
  there is an equivalence
  \begin{equation}
    \label{eq:excess-inter-categ}
    \begin{split}
      u^{\ast}u_{\ast}\shname{M}
      &=\shname{M}\otimes_{\shname{O}_{T}}
      \bigwedge\nolimits^{\bullet}\restr{\shname{F}^{\vee}}_{T}\text{.}
    \end{split}
  \end{equation}
\end{corlr}

\begin{proof}
  The $\infty$-functor $u^{\ast}$ is a left-adjoint so it preserves
  colimits, among which in particular cofibres. By definition, we are
  given an equivalence $u^{\ast}\widetilde{s}\simeq u^{\ast}0=0$, so
  the image by $u^{\ast}$ of~\cref{eq:zerosec-struct-shf-formula}
  takes the form $\shname{Sym}(\cofib0)/(t-1)$. By definition of the
  zero morphism, we may decompose this pushout as the composite of two
  amalgamated sums:
  \begin{equation}
    \label{eq:cofib-zero-nlin-diagr}
    \begin{tikzcd}
      \shname{F}^{\vee}[1]\oplus\shname{O}_{M} \arrow[dr,phantom,very
      near start,"\lrcorner"] & \shname{F}^{\vee}[1] \arrow[l]
      \arrow[dr,phantom,very near start,"\lrcorner"] &
      0 \arrow[l,"!"'] \\
      \shname{O}_{M} \arrow[u] & 0 \arrow[l,"!"'] \arrow[u,"!"] &
      \shname{F}^{\vee} \arrow[l,"!"']  \arrow[u,"!"'] \arrow[ll,bend
      left,"0"]
    \end{tikzcd}\text{,}
  \end{equation}
  so that
  $\shname{Sym}_{\shname{O}_{M}}(\cofib{0})
  =\shname{Sym}_{\shname{O}_{M}}(\shname{F}^{\vee}[1]\oplus\shname{O}_{M})
  =\shname{Sym}(\shname{F}[-1]^{\vee})
  \otimes_{\shname{O}_{M}}\shname{O}_{M}\{t\}$. As $u^{\ast}$ has a
  structure of monoidal $\infty$-functor, this extends to any
  $\shname{O}_{T}$-module $\shname{M}$ in the image of $u^{\ast}$.

  Of course, this can also be obtained more directly from the fact
  that, when $s$ is restricted to zero, the leftmost diagram below is
  the image by $\shname{Spec}_{M}\shname{Sym}_{\shname{O}_{M}}$ of the
  rightmost one:
  \begin{equation}
    \label{eq:zero-self-inter-cofib-zero}
    \begin{tikzcd}
      T \arrow[d,"u"'] \arrow[r,"u"] \arrow[dr,phantom,very near
      start,"\lrcorner"] & M \arrow[d,"0_{\mathbb{V}_{M}(\shname{F})}"] \\
      M \arrow[r,"0_{\mathbb{V}_{M}(\shname{F})}"'] &
      \mathbb{V}_{M}(\shname{F})
    \end{tikzcd}\qquad
    \begin{tikzcd}
      u_{\ast}\shname{O}_{T} \arrow[dr,phantom,very near
      start,"\lrcorner"] & 0 \arrow[l] \\
      0 \arrow[u] & \shname{F}^{\vee} \arrow[l,"!"] \arrow[u,"!"']
    \end{tikzcd}\text{.}
  \end{equation}
\end{proof}

\begin{remark}[Lie-theoretic interpretation]
  \label{rmrk:excess-inter-lie-envlop}
  The excess intersection formula can also be seen as coming from the
  study of the $\mathcal{L}_{\infty}$-algebroid associated with the
  closed embedding $u$. Indeed, we are studying the geometry of a
  closed sub-derived stack $T\subset M$, which can be understood
  through that of its formal neighbourhood
  $\widehat{M}_{T}=M\times_{M_{\mathrm{dR}}}T_{\mathrm{dR}}$. This is
  a formally algebraic derived stack (see \cite[section
  4.1]{calaque18:_formal} or \cite[Chapter 1, Definition
  7.1.2]{gaitsgory17b} for details) which is a formal thickening of
  $T$. By \cite[Chapter 5, Theorem 2.3.2]{gaitsgory17b}, the
  $\infty$-category of formal thickenings of $T$ is equivalent to that
  of groupoid objects in formally algebraic derived stacks over $T$
  (\emph{via} the $\infty$-functor sending a thickening
  $T\to\shname{F}$ to its simplicial kernel, or \v{C}ech nerve), and
  following the philosophy of formal moduli problems it can be
  considered as a model for the $\infty$-category of
  $\mathcal{L}_{\infty}$-algebroids.

  We have the sequence of adjunctions
  $u^{\ast}\dashv u_{\ast}\dashv u^{!}$, implying that the comonad
  $u^{\ast}u_{\ast}$ is left-adjoint to the monad $u^{!}u_{\ast}$ (on
  $\operatorname{Ind}(\cat{Coh}^{\mathrm{b}}(T))$, only
  $u^{\ast}u_{\ast}$ restricting to a comonad on
  $\cat{Coh}^{\mathrm{b}}(T)$ when $u$ is quasi-smooth
  by~\cite[Chapter 4., Lemma 3.1.3]{gaitsgory17}). Let us write
  $T\xrightarrow{\widehat{u}}\widehat{M}_{T}\xrightarrow{p}M$ the
  factorisation of $u$, so that
  $u^{!}u_{\ast}=\widehat{u}^{!}p^{!}p_{\ast}\widehat{u}_{\ast}$. Note
  that $p\colon M\times_{M_{\mathrm{dR}}}T_{\mathrm{dR}}\to M$ is the
  canonical projection, and as both $T_{\mathrm{dR}}$ and
  $M_{\mathrm{dR}}$ are étale over $\spec{\Bbbk}$ it is also an étale
  morphism, and we recover $\widehat{u}^{!}\widehat{u}_{\ast}$.
  Following~\cite[Chapter 8, 4.1.2]{gaitsgory17b}, the monad
  $u^{!}u_{\ast}$ becomes the universal enveloping algebra of the
  $\mathcal{L}_{\infty}$-algebroid associated with $u$, endowed with
  the Poincaré--Birkhoff--Witt filtration. As the $\infty$-functor of
  assciated graded is conservative when restricted to (co)connective
  filtrations, we only need an expression for the associated graded of
  the PBW filtration. The result is then nothing but the PBW
  isomorphism of~\cite[Chapter 9, Theorem 6.1.2]{gaitsgory17b} stating
  that for any regular embedding of derived stacks
  $u\colon T\hookrightarrow M$, the monad
  $\widehat{u}^{!}\widehat{u}_{\ast}$ on
  $\operatorname{Ind}(\cat{Coh}^{\mathrm{b}}(T))$ is equivalent to
  tensoring by
  $\shname{Sym}_{\shname{O}_{T}}(\mathbb{T}_{\widehat{u}})$, and
  $\mathbb{T}_{\widehat{u}}=\mathbb{T}_{u}$ since $p$ is
  étale. Passing back to the adjoint, we do obtain that
  $u^{\ast}u_{\ast}$ is equivalent to tensoring with
  $\shname{Sym}_{\shname{O}_{T}}(\mathbb{T}_{u}^{\vee})$.

  A similar equivalence between the Hopf comonad $u^{\ast}u_{\ast}$
  and tensoring by the jet algebra (the dual of the universal
  enveloping algebra) of $\mathbb{T}_{u}$ was established in
  \cite[Theorem 1.3]{calaque14:_lie} using the model of dg-Lie
  algebroids for $\mathcal{L}_{\infty}$-algebroids (see
  \cite[Proposition 4.3, Theorem 4.11]{calaque18:_formal} for a
  precise statement of the equivalence between dg-Lie algebroids and
  formally algebraic derived stacks as models for
  $\mathcal{L}_{\infty}$-algebroids). However this approach does not
  provide the PBW theorem needed to identify the jet algebra of
  $\mathbb{T}_{u}$ with $\shname{Sym}(\mathbb{L}_{u})$.

  Finally, it is easy to see from~\cref{pros:cotgt-cplx-vect-bndl}
  that the base-change property of cotangent complexes and the fibre
  sequence associated with the composition $\varpi\circ s=\id$ imply
  $\mathbb{L}_{u}=u^{\ast}\mathbb{L}_{s}
  =u^{\ast}s^{\ast}\mathbb{L}_{\varpi}[1]=u^{\ast}\shname{F}^{\vee}[1]$.

  Then, conservativity of the restriction of $u^{!}$ to the
  $\infty$-category
  $\operatorname{Ind}(\cat{Coh}^{\mathrm{b}}(M)_{T})
  \simeq\operatorname{Ind}(\cat{Coh}^{\mathrm{b}}(\widehat{M}_{T}))$
  of coherent sheaves with support (by~\cite[Chapter 4, Proposition
  6.1.3 (c)]{gaitsgory17}) gives another reason for the equivalence
  $u_{\ast}\shname{O}_{Z}\simeq(\cofib(\widetilde{s}))/(t-1)$.
\end{remark}

Although it is not possible to directly relate $s$ and the zero
section at the geometric level and to obtain an expression of
$u_{\ast}\shname{O}_{T}$ in terms of the Euler class of
$\shname{F}^{\vee}$, passing to $G$-theory a homotopy between the maps
they induce always does exist, and hence we recover the classical
formulation of the quantum Lefschetz hyperplane formula.

\begin{warning}
  \label{remark:boundedness}
  When $T$ is quasi-smooth, $\shname{O}_{T}$ belongs to
  $\cat{Coh}^{\mathrm{b}}(T)$ so by~\cite[Chapter 4, Lemma
  5.1.4]{gaitsgory17} $u_{\ast}\shname{O}_{T}$ is in
  $\cat{Coh}^{\mathrm{b}}(M)$ and thus defines a class in
  $G_{0}(M)$.

  However this is no longer the case if $T$ is not quasi-smooth (or,
  more generally, when the embedding $u\colon T\hookrightarrow M$ is
  not quasi-smooth even if $T$ itself is); for example when $s=0$,
  $\bigwedge^{\bullet}\shname{F}^{\vee}$ will fail to be bounded if
  $\shname{F}$ does not have Tor-amplitude concentrated in degree $0$.
\end{warning}

We recall the notation of the $G$-theoretic Euler class of a locally
free $\shname{O}_{M}$-module $\shname{G}$ of finite rank:
$\lambda_{-1}(\shname{G})
\coloneqq\bigl[\bigwedge^{\bullet}\shname{G}\bigr]
=\sum_{i\geq0}\bigl[\bigwedge^{i}\shname{G}[i]\bigr] =\sum_{i}(-1)^{i}
\bigl[\bigwedge^{i}\shname{G}\bigr]\in G_{0}(M)$.
  
\begin{corlr}[{\cite[Lemma 2.1]{khan19:_excess}}]
  \label{corlr:zero-sec-k-thry}
  Suppose $\shname{F}$ is a vector bundle. There is an equivalence of
  $G$-theory operators
  \begin{equation}
    \label{eq:zero-sec-k-thry-formula}
    u_{\ast}u^{\ast}
    \simeq(-)\otimes\lambda_{-1}(\shname{F}^{\vee})
    \colon G(M)\to G(M)\text{.}
  \end{equation}
\end{corlr}

\begin{proof}
  We first note that, by definition, $\shname{F}$ being locally free
  of finite rank means that it is (flat-locally) almost perfect, which
  makes it bounded, and flat, which makes it of Tor-amplitude
  concentrated in $[0]$ and implies that $\shname{F}^{\vee}[1]$ has
  Tor-amplitude in $[-1,0]$ so that its symmetric algebra is still
  bounded and thus in $\cat{Coh}^{\mathrm{b}}(M)$, defining an element
  of $G_{0}(M)$.

  By~\cite[Lemma 1.3]{khan19:_excess}, the fibre sequence
  $\shname{O}_{M}\to\cofib(\widetilde{s})\to\shname{F}^{\vee}[1]$
  implies that
  $[\shname{Sym}_{\shname{O}_{M}}^{n}(\cofib\widetilde{s})]
  =\oplus_{i=0}^{n}[\shname{Sym}^{n-i}(\shname{O}_{M})
  \otimes\shname{Sym}^{i}(\shname{F}^{\vee}[1])]$ for all $n\geq0$. By
  the $\mathbb{A}^{1}$-invariance of $G$-theory we may remove the
  symmetric algebra of $\shname{O}_{M}$, which gives the result.
\end{proof}

\section{The geometric Lefschetz principle}
\label{sec:geom-lefsch-princ}

\subsection{Review of the derived moduli stack of stable maps}
\label{sec:revi-deriv-moduli}

Let $X$ be a target derived $1$-Artin stack. We denote
$\pi_{g,n}\colon\stkcurv[g]{n}\to\modstk[g]{n}$ --- omitting mention of
the twisted structure --- the universal curve over the moduli stack of
prestable stacky curves of genus $g$ with $n$ markings (and arbitrary
orders of isotropy groups at the markings).

\begin{remark}
  \label{rmrk:stbl-curv-derived-target}
  Note that we can allow $X$ to be derived --- although it is still
  required to be only $1$-algebraic, as otherwise twisted curves will
  not be enough to ensure properness of the stack of stable maps to it
  --- without any change to the usual theories of stable maps to $X$,
  as the moduli problem for prestable curves parametrises \emph{flat}
  families, whose fibres over a derived stack must still be
  classical. More precisely, \cite[Theorem
  8.1.3]{lurie04:_deriv_algeb_geomet} shows (see also
  \cite[Proposition 4.5]{porta20:_non_k} for a precise proof of the
  non-archimedean analogue) that the obvious extension of the moduli
  problem for prestable curves to a derived moduli problem is
  representable by a classical DM stack $\modstk[g]{n}$.
\end{remark}

The $\infty$-category of derived stacks (or any of its slices), as an
$\infty$-topos, is also cartesian closed, with internal hom denoted
$\rmap(-,-)$; the property of being right-adjoint to the cartesian
product imposes that, as a functor of points, for any base $B$ and
$B$-stacks $M$ and $N$, the $B$-stack $\rmap_{/B}(M,N)$ be given by
\begin{equation}
  \label{eq:map-stack-fctr-pts}
  \rmap_{/B}(M,N)\colon(T\to
  B)\mapsto\hom_{\cat{dSt}_{/B}}(M\times_{B}T,N)\text{.}
\end{equation}

\begin{pros}[{\cite[(4.3.4)]{mann18:_brane_gromov_witten_k},
    \cite[Proposition 5.1.10]{halpern-leistner14:_mappin},
    \cite[Proposition 19.1.4.1 (2)]{lurie19:_spect_algeb_geomet}}]
  \label{pros:map-stk-tgt-cplx}
  Let $M$ be a base derived stack and $C\to M$ and $D\to M$ be two
  $M$-derived stacks. Then
  \begin{equation}
    \label{eq:map-stk-tgt-cplx}
    \mathbb{T}_{\rmap_{/M}(C,D)/M}=\varpi_{\ast}\ev^{\ast}\mathbb{T}_{D/M}
  \end{equation}
  where $\varpi\colon C\times_{M}\rmap_{/M}(C,D)\to\rmap_{/M}(C,D)$ is
  the projection and $\ev\colon\rmap_{/M}(C,D)\to D$ is the evaluation
  map.
\end{pros}

\begin{remark}
  \label{rmrk:cotgt-cplx-obstr-thry}
  In the case of the open moduli substack of stable maps (recall that
  a Zariski-open immersion, like any étale map, has vanishing relative
  cotangent complex), we recognise in~\cref{eq:map-stk-tgt-cplx} the
  formula defining the perfect obstruction theory used to define the
  virtual fundamental class in Gromov--Witten theory,
  \emph{cf}~\cite[Proposition 4.3.1]{mann18:_brane_gromov_witten_k}.
\end{remark}

\begin{corlr}
  \label{corlr:stbl-map-mod-qsm}
  If $X$ is smooth (resp. smooth with convex tangent bundle), then the
  derived stack
  $\rmap_{/\modstk[g]{n}}\left(\stkcurv[g]{n},X\times\modstk[g]{n}\right)$
  is a quasi-smooth (resp. smooth). \qed{}
\end{corlr}

\begin{remark}
  \label{rmrk:truncat-rmap-thicken}
  For any classical scheme $T\to\modstk[g]{n}$, we can compute that
  \begin{equation}
    \label{eq:rmap-thicken}
    \begin{split}
      \bigl(\truncat\rmap_{/\modstk[g]{n}}\left(\stkcurv[g]{n},
        X\times\modstk[g]{n}\right)\bigr)& (T\to\modstk[g]{n})\\
      &\!\!=\hom_{\cat{dSt}}\bigl(\stkcurv[g]{n}
      \lmtimes_{\modstk[g]{n}}T, X\bigr)\\
      &\!\!\simeq\hom_{\cat{dSt}}\bigl(\stkcurv[g]{n}
      \lmtimes_{\modstk[g]{n}}^{\func{t}}T,
      X\bigr)\\
      &\!\simeq\hom_{\cat{St}}\bigl(\stkcurv[g]{n}
      \lmtimes_{\modstk[g]{n}}^{\func{t}}T,
      \truncat{X}\bigr)\\
      &\!\!=\bigl(\shname{Map}_{/\modstk[g]{n}}\left(\stkcurv[g]{n},
        \truncat{X}\times\modstk[g]{n}\right)\bigr)
      (T\to\modstk[g]{n})
    \end{split}
  \end{equation}
  where the first isomorphism is because
  $\stkcurv[g]{n}\to\modstk[g]{n}$ is flat and the second from the
  right-adjoint property of $\truncat$. This shows that
  $\rmap_{/\modstk[g]{n}}\left(\stkcurv[g]{n},X\times\modstk[g]{n}\right)$
  is a derived thickening of the classical mapping stack
  $\shname{Map}_{/\modstk[g]{n}}\left(\stkcurv[g]{n},
    X\times\modstk[g]{n}\right)$ (see also~\cite[Theorem 2.2.6.11,
  hypothesis (1)]{toen08:_homot}).
\end{remark}

We can view
$\rmap_{/\modstk[g]{n}}\left(\stkcurv[g]{n},X\times\modstk[g]{n}\right)$
as a (derived) moduli stack for families of \emph{pre}stable maps to
$X$; in particular, it is only Artin and not Deligne--Mumford. Since
our reasoning for proving the quantum Lefschetz principle is purely
formal, it will mainly work at the level of these general mapping
stacks. However, we are interested in a more geometric subclass of
maps, which only have finite automorphisms and define a
Deligne--Mumford substack: this comes down to imposing a stability
condition on the maps.

Since our result holds for stacky targets as well as schematic ones,
we will use an adapted stability condition, inspired by the quasimap
stability condition of~\cite{ciocan-fontanine15:_orbif} for global
quotient orbifolds, and laid out for this generality in some more
detail in~\cite[\S{4.2.1.1}]{kern21:_categ_delig}.

\begin{construct}[Stability condition]
  \label{constr:stab-cond-def}
  Let
  $\shname{L}=\shname{L}_{0}\otimes\varepsilon
  \in\operatorname{Pic}(X)\otimes_{\mathbb{Z}}\mathbb{Q}$ be a line
  bundle with $\varepsilon$ positive.

  Say, following~\cite{heinloth18:_hilber_mumfor} as a simplified
  version of the criterion of~\cite{halpern-leistner18}, that a point
  $x$ of $X$ is $\shname{L}_{0}$-stable (or equivalently, since
  $\varepsilon>0$, $\shname{L}$-stable) if for any map
  $f\colon[\mathbb{A}^{1}/\mathbb{G}_{\mathrm{m}}]\to X$ such that
  $f(0)\neq f(1)$, the weight of the $\mathbb{G}_{\mathrm{m}}$-action
  on $f(0)^{\ast}\shname{L}_{0}$ (a quasicoherent sheaf on
  $\{0\}\simeq[\ast/\mathbb{G}_{\mathrm{m}}]
  \subset[\mathbb{A}^{1}/\mathbb{G}_{\mathrm{m}}]$, viewed as a
  $\mathbb{G}_{\mathrm{m}}$-equivariant module) is negative. We will
  also require (as in~\cite{heinloth18:_hilber_mumfor}) that the
  stable points have finite isotropy groups, so that the
  $\shname{L}$-stable locus defines a Deligne--Mumford substack
  $X^{\text{$\shname{L}$-st}}\subset X$.

  If $(C;\Sigma_{1},\dots,\Sigma_{n})$ is an $n$-pointed stacky curve,
  a representable morphism $C\to X$ is said to be
  \textbf{pre-$\shname{L}$-quasistable} if it maps the generic point
  of any irreducible component of $C$ to $X^{\text{$\shname{L}$-st}}$
  (so that it has only a finite number of basepoints), and its
  basepoints are disjoint from the special points of $C$.

  Since $X^{\text{$\shname{L}$-st}}$ is a DM stack, it makes sense to
  require in addition that the restriction of $\shname{L}_{0}$ be
  ample, and we do so. We can now further say that $f\colon C\to X$ is
  \textbf{$\shname{L}$-quasistable} if
  \begin{enumerate}
  \item letting $e$ denote the least common multiple of the
    $\operatorname{ord}(\shname{Aut}(x))$ for $x$ points of $X$ and
    $\uly{C}$ the coarse moduli space of $C$, we have that
    \begin{equation}
      \label{eq:stab-cond}
      \omega_{\uly{C}}\left(\sum_{i=1}^{n}\uly{\Sigma_{i}}\right)
      \otimes(f^{\ast}\shname{L}_{0}^{e})^{\varepsilon/e}
    \end{equation}
    is ample;
  \item for any point $p$ of $C$,
    $\varepsilon\cdot\operatorname{lgth}_{f}(p)\leq1$, where
    $\operatorname{lgth}_{f}(p)$ is the order of contact of $f$ with
    the unstable locus of $X$ at $p$.
  \end{enumerate}
\end{construct}

The quasistability condition for a map from a stable curve to $X$,
that is a point of
$\shname{Map}_{/\modstk[g]{n}}\left(
  \stkcurv[g]{n},X\times\modstk[g]{n}\right)$, is open, and thus
defines, for any target class $\beta$ on $X$, an open substack which
we denote $\compmodsp[g]{n}(X,\beta)$ --- leaving the polarisation
$\shname{L}$ implicit since it will not play a role in the results of
this paper.

Now, by~\cite[Corollary 2.2.2.10]{toen08:_homot} the (small) Zariski
$\infty$-sites of a derived stack $M$ and of its truncation
$\truncat{M}$ are equivalent (and in particular $1$-sites). It ensues
as in~\cite[Proposition 2.1]{schurg15:_deriv} that any open substack
$U$ of $\truncat{M}$ lifts uniquely to an open sub-derived stack
$\mathbb{R}U\subset M$ such that $\mathbb{R}U\times_{M}\truncat{M}=U$
(so in particular $\mathbb{R}U$ is a derived thickening of $U$).

\begin{definition}
  \label{def:mod-dstk-stbl-map}
  The \textbf{derived moduli stack
    $\mathbb{R}\compmodsp[g]{n}(X,\beta)$} of genus-$g$, $n$-pointed
  \textbf{stable quasimaps to $X$} of class $\beta$ is the open
  sub-derived stack of
  $\rmap_{/\modstk[g]{n}}\left(\stkcurv[g]{n},X\times\modstk[g]{n}\right)$
  corresponding to the open substack
  $\compmodsp[g]{n}(X,\beta)
  \subset\shname{Map}_{/\modstk[g]{n}}\left(
    \stkcurv[g]{n},X\times\modstk[g]{n}\right)$.
\end{definition}

By~\cref{rmrk:cotgt-cplx-obstr-thry}
and~\cref{corlr:stbl-map-mod-qsm}, when $X^{\text{st}}$ is smooth,
$\mathbb{R}\compmodsp[g]{n}(X,\beta)$ enhances to a (quasi-smooth)
derived geometric object the data of the moduli stack
$\compmodsp[g]{n}(X,\beta)$ and its perfect obstruction theory.

\subsection{Identification of the derived moduli stacks}
\label{sec:ident-deriv-moduli}

Let $X$ be an Artin derived stack and
$\shname{E}\in\cat{Perf}(\shname{O}_{X})$ a perfect
$\shname{O}_{X}$-module, giving the perfect cone
$E=\mathbb{V}_{X}(\shname{E})$. Let $s$ be a section of $E$, and
denote
\begin{equation}
  \label{eq:def-Z-zero-loc}
  Z=X\lmtimes_{s,E,0}X\subset X
\end{equation}
its (derived) zero locus.

For a fixed morphism $\pi\colon\mathfrak{C}\to\mathfrak{M}$ of derived
$\Bbbk$-stacks \emph{proper and of finite cohomological dimension}, we
consider the universal map from a base-change of $\mathfrak{C}$ over
the derived mapping $\mathfrak{M}$-stack
$\rmap_{/\mathfrak{M}}\left(\mathfrak{C},X\times\mathfrak{M}\right)$:
\begin{equation}
  \label{eq:univ-diagr-map-stk}
  \begin{tikzcd}
    \mathfrak{C}\lmtimes_{\mathfrak{M}}\rmap_{/\mathfrak{M}}
    \left(\mathfrak{C},X\times\mathfrak{M}\right)
    \arrow[d,"\func{p}"'] \arrow[r,"\ev"]
    & X \\
    \rmap_{/\mathfrak{M}}\left(\mathfrak{C},X\times\mathfrak{M}\right)
    &
  \end{tikzcd}\text{.}
\end{equation}

Let
$\mathbb{E}\coloneqq\func{p}_{\ast}\ev^{\ast}E
=\mathbb{V}_{\rmap_{/\mathfrak{M}}\left(\mathfrak{C},X\times\mathfrak{M}\right)}(
\func{p}_{\ast}\ev^{\ast}\shname{E})$ be the induced abelian (and
perfect by~\cref{rmrk:qsm-preserves-perf} if $\func{p}$ is
quasi-smooth) cone over
$\rmap_{/\mathfrak{M}}\left(\mathfrak{C},X\times\mathfrak{M}\right)$,
and $\sigma\coloneqq\func{p}_{\ast}\ev^{\ast}s$ its induced
section. Write also
$0_{\mathbb{E}}\colon\rmap_{/\mathfrak{M}}\left(\mathfrak{C},
  X\times\mathfrak{M}\right)\to\mathbb{E}$ for the zero section.

\begin{thm}
  \label{thm:identf-rmap-stk}
  There is an equivalence of
  $\rmap_{/\mathfrak{M}}(\mathfrak{C},X\times\mathfrak{M})$-derived
  stacks
  \begin{equation}
    \label{eq:compar-rmap-stk}
    \rmap_{\mathfrak{M}}\left(\mathfrak{C},Z\times\mathfrak{M}\right)
    \simeq\rmap_{/\mathfrak{M}}\left(\mathfrak{C},X\times\mathfrak{M}\right)
    \mathop{\times}_{\sigma,\mathbb{E},0_{\mathbb{E}}}
    \rmap_{/\mathfrak{M}}\left(
      \mathfrak{C},X\times\mathfrak{M}\right)\text{,}
  \end{equation}
  that is the diagram
  \begin{equation}
    \label{eq:main-rslt-abstr-cart-diagr}
    \begin{tikzcd}
      \rmap_{\mathfrak{M}}\left(\mathfrak{C},Z\times\mathfrak{M}\right)
      \arrow[d,"u_{2}"'] \arrow[r,"u_{1}"] \arrow[dr,phantom,very near
      start,"\lrcorner"] &
      \rmap_{/\mathfrak{M}}\left(\mathfrak{C},X\times\mathfrak{M}\right)
      \arrow[d,"\sigma=\func{p}_{\ast}\ev^{\ast}s"] \\
      \rmap_{/\mathfrak{M}}\left(\mathfrak{C},X\times\mathfrak{M}\right)
      \arrow[r,"0_{\mathbb{E}}"'] & \mathbb{E}=\func{p}_{\ast}\ev^{\ast}E
    \end{tikzcd}
  \end{equation}
  is cartesian.
\end{thm}

The~\namecref{thm:identf-rmap-stk} will follow directly from some
formal results.

% \label{rmrk:hom-preserv-lim-adj}
We first note that, since we work in the $\infty$-category
$\cat{dSt}_{\Bbbk}$, which as an $\infty$-topos is cartesian closed, its
the internal hom $\infty$-functor $\rmap(-,-)$ is a right-adjoint to
taking cartesian product, so by~\cite[Theorem 2.4.2]{riehl21:_elemen}
preserves limits.

\begin{remark}
  This limit preservation property is also due to the more conceptual
  reason that, in enriched $\infty$-categories, limits and colimits
  can be defined representably, \emph{cf.}~\cite[Proposition
  1.1.2.2.8, Example 1.1.2.2.9]{kern21:_categ_delig}.
\end{remark}

Applying this to our case, we find that
$\rmap_{/\mathfrak{M}}\left(\mathfrak{C},Z\times\mathfrak{M}\right)$
is equivalent to the fibre product
\begin{equation}
  \label{eq:compar-rmap-stk-step-pullback}
  \rmap_{/\mathfrak{M}}\left(\mathfrak{C},X\times\mathfrak{M}\right)
  \mathop{\times}_{\rmap_{/\mathfrak{M}}\left(\mathfrak{C},E\times\mathfrak{M}\right)}
  \rmap_{/\mathfrak{M}}\left(\mathfrak{C},X\times\mathfrak{M}\right)\text{,}
\end{equation}
with structure morphisms induced by $s$ and the zero section of
$E$. Hence, in order to prove~\cref{thm:identf-rmap-stk} we only need
to identify the two derived stacks over which the fibre products are
taken (as well as the two pairs of structure maps), the derived stack
of maps to the abelian cone $E$ and the induced cone
$\mathbb{E}=\func{p}_{\ast}\ev^{\ast}E$.

\begin{pros}
  \label{pros:perfcone-equiv-rmap}
  There is an equivalence of
  $\rmap_{/\mathfrak{M}}\left(\mathfrak{C},
    X\times\mathfrak{M}\right)$-derived stacks
  \begin{equation}
    \label{eq:compar-perfcone-rmap}
    \rmap_{/\mathfrak{M}}\left(\mathfrak{C},E\times\mathfrak{M}\right)
    \simeq\mathbb{E}\text{.}
  \end{equation}
\end{pros}

\begin{proof}
  Let
  $a\colon
  S\to\rmap_{/\mathfrak{M}}\left(\mathfrak{C},X\times\mathfrak{M}\right)$
  be an
  $\rmap_{/\mathfrak{M}}\left(\mathfrak{C},X\times\mathfrak{M}\right)$-stack,
  with corresponding family
  $C_{a}=a^{\ast}\mathfrak{C}=S\times_{\mathfrak{M}}\mathfrak{C}\to S$
  (where we implicitly push the structure maps forward along
  $\rmap_{/\mathfrak{M}}\left(\mathfrak{C},X\times\mathfrak{M}\right)
  \to\mathfrak{M}$).  Note that, as
  $\func{p}\colon\mathfrak{C}\times_{\mathfrak{M}}\rmap_{/\mathfrak{M}}
  \left(\mathfrak{C},X\times\mathfrak{M}\right)\to\rmap_{/\mathfrak{M}}
  \left(\mathfrak{C},X\times\mathfrak{M}\right)$ is just projection
  onto the first factor, we have
  \begin{equation}
    \label{eq:basechg-curv-rmap}
    \begin{split}
      \func{p}^{-1}(a)&
      =S\times_{\rmap_{/\mathfrak{M}}\left(\mathfrak{C},X\times\mathfrak{M}\right)}
      \left(\rmap_{/\mathfrak{M}}\left(
          \mathfrak{C},X\times\mathfrak{M}\right)\times_{\mathfrak{M}}
        \mathfrak{C}\right)\\
      &=S\times_{\mathfrak{M}}\mathfrak{C}\eqqcolon C_{a}\text{,}
    \end{split}
  \end{equation} as seen in the cartesian diagram
  \begin{equation}
    \label{eq:basechg-curv-rmap-diagr}
    \begin{tikzcd}
      C_{a}=S\lmtimes_{\mathfrak{M}}\mathfrak{C} \arrow[d]
      \arrow[r,"\widetilde{a}"] \arrow[dr,phantom,very near
      start,"\lrcorner"] & \rmap_{/\mathfrak{M}}\left(\mathfrak{C},
        X\times\mathfrak{M}\right)\lmtimes_{\mathfrak{M}}\mathfrak{C}
      \arrow[d,"\func{p}"] \arrow[r] \arrow[dr,phantom,very near
      start,"\lrcorner"] & \mathfrak{C} \arrow[d] \\
      S \arrow[r,"a"'] & \rmap_{/\mathfrak{M}}\left(\mathfrak{C},
        X\times\mathfrak{M}\right) \arrow[r] & \mathfrak{M}
    \end{tikzcd}\text{.}
  \end{equation}

  By~\cref{lemma:ab-cone-nat}, as
  $\pi\colon\mathfrak{C}\to\mathfrak{M}$ was supposed of finite
  cohomological dimension and morphisms of finite cohomological
  dimension are stable by base-change, we have
  \begin{equation}
    \label{eq:vbnld-rmap-value-comput}
    \begin{split}
      \mathbb{E}(a)
      &=\mathbb{V}_{\rmap_{/\mathfrak{M}}\left(\mathfrak{C},X\times\mathfrak{M}\right)}(
      \func{p}_{\ast}\ev^{\ast}\shname{E})(a)\\
      &=\func{p}_{\ast}\mathbb{V}
      _{\mathfrak{C}\times_{\mathfrak{M}}\rmap_{/\mathfrak{M}}\left(\mathfrak{C},X\times\mathfrak{M}\right)}(
      \ev^{\ast}\shname{E})(a)\\
      &=\mathbb{V}
      _{\mathfrak{C}\times_{\mathfrak{M}}\rmap_{/\mathfrak{M}}\left(\mathfrak{C},X\times\mathfrak{M}\right)}(
      \ev^{\ast}\shname{E})(C_{a})\\
      &=\map_{\cat{Perf}(\shname{O}_{C_{a}})}\left(\shname{O}_{C_{a}},
        \widetilde{a}^{\ast}\ev^{\ast}\shname{E}\right)
      =\map_{/X}(C_{a},E)\text{,}
    \end{split}
  \end{equation}
  where $\ev\circ\widetilde{a}\colon C_{a}\to X$ is the map from a
  family of curves to $X$ classified by $a$.
  
  Meanwhile, we have by definition
  \begin{equation}
    \label{eq:rmap-vb-value-comput}
    \begin{split}
      &\rmap_{/\mathfrak{M}}\left(
        \mathfrak{C},E\times\mathfrak{M}\right)(a)\\
      =&\map_{\rmap_{/\mathfrak{M}}\left(\mathfrak{C},X\times\mathfrak{M}\right)}(S,
      \rmap_{/\mathfrak{M}}\left(\mathfrak{C},E\times\mathfrak{M}\right))\\
      \simeq&\map_{/\mathfrak{M}}(S,
      \rmap_{/\mathfrak{M}}\left(\mathfrak{C},E\times\mathfrak{M}\right))
      \lmtimes
      _{\map_{/\mathfrak{M}}(S,\rmap_{/\mathfrak{M}}\left(\mathfrak{C},X\times\mathfrak{M}\right))}
      \{a\}\text{.}
          \end{split}
  \end{equation}
  Indeed\footnote{This argument was suggested to the author by
    Benjamin
    Hennion}, % from\cite[Lemma 6.1.3.13]{lurie09:_higher} the standard
  % categorical arguments show that
  \cite[Lemma 5.5.5.12]{lurie09:_higher} shows that for any
  morphism %of derived stacks
  $p\colon M^{\prime}\to M$ in an $\infty$-category and any cospan
  $S\to M^{\prime}\leftarrow T$ over $M^{\prime}$ we have
  $\map_{/M^{\prime}}(S,T)\simeq\map_{/M}(S,T)\times_{\map_{/M}(S,M^{\prime})}\{p\}$
  ; and we can compute
  \begin{equation}
    \label{eq:rmap-vb-comput-end}
    \begin{split}
      &\rmap_{/\mathfrak{M}}\left(\mathfrak{C},
        E\times\mathfrak{M}\right)(a)\\
      \simeq&\map_{/\mathfrak{M}}(S\times_{\mathfrak{M}}\mathfrak{C},
      E\times\mathfrak{M})
      \lmtimes
      _{\map_{/\mathfrak{M}}(S\times_{\mathfrak{M}}\mathfrak{C},X\times\mathfrak{M})}
      \{a\}\\
      =&\map_{/X\times\mathfrak{M}}(C_{a},E\times\mathfrak{M})
      =\map_{/X}(C_{a},E)\text{.}
    \end{split}
  \end{equation}
\end{proof}

This completes the proof of~\cref{thm:identf-rmap-stk}.\qed{}
\linebreak[4]

In our setting, we will have $\mathfrak{M}=\modstk[g]{n}$,
$\mathfrak{C}=\stkcurv[g]{n}$, and the morphism of finite
cohomological dimension $\pi\colon\mathfrak{C}\to\mathfrak{M}$ is
$\pi_{g,n}$ the universal curve over the moduli stack of prestable
stacky curves of genus $g$ with $n$ marked
points. % One may also wish to
% replace $\modstk[g]{n}$ by an appropriate moduli stack of twisted
% curves to accommodate for orbifold targets; however we have not done
% so as we believe that the current notion of morphism of twisted curves
% is not suitable for the derived stacks we need as targets.

We will also write $\func{p}_{g,n}=\func{p}$, $\ev_{g,n}=\ev$ and
$\mathbb{E}_{g,n}=\mathbb{E}=(\func{p}_{g,n})_{\ast}\ev_{g,n}^{\ast}E$.

\begin{corlr}[Geometric quantum Lefschetz principle]
  \label{corlr:main-result}
  Suppose $X$ satisfies the conditions required
  for~\cref{constr:stab-cond-def} (that is, $X$ is $1$-Artin and has a
  line bundle $\shname{L}_{0}$ whose stable locus is
  $1$-Deligne--Mumford and quasiprojective with $\shname{L}_{0}$ as
  ample polarisation), and fix a class $\beta\in A_{1}X$.  There is an
  equivalence of
  $\rmap_{/\modstk[g]{n}}\left(\stkcurv[g]{n},
    Z\times\modstk[g]{n}\right)$-derived stacks
  \begin{equation}
    \label{eq:main-rslt-compar-stabl-map-vb}
    \coprod_{i_{\ast}\gamma=\beta}\mathbb{R}\compmodsp[g]{n}(Z,\gamma)
    \simeq\mathbb{R}\compmodsp[g]{n}(X,\beta)
    \mathop{\times}_{\restr{\mathbb{E}}_{\mathbb{R}\compmodsp[g]{n}(X,\beta)}}
    \mathbb{R}\compmodsp[g]{n}(X,\beta)\text{.}
  \end{equation}
  % that is the diagram
  % \begin{equation}
  %   \label{eq:main-rslt-cart-diagr}
  %   \begin{tikzcd}
  %     \coprod\limits_{i_{\ast}\gamma=\beta}\mathbb{R}\compmodsp[g]{n}(Z,\gamma)
  %     \arrow[d,"u_{2}"'] \arrow[r,"u_{1}"] \arrow[dr,phantom,very near
  %     start,"\lrcorner"] & \mathbb{R}\compmodsp[g]{n}(X,\beta)
  %     \arrow[d,"\func{p}_{g,n;\ast}\ev_{g,n}^{\ast}s"] \\
  %     \mathbb{R}\compmodsp[g]{n}(X,\beta) \arrow[r,"0_{\mathbb{E}}"']
  %     & \restr{\mathbb{E}_{g,n}}_{\mathbb{R}\compmodsp[g]{n}(X,\beta)}
  %   \end{tikzcd}
  % \end{equation}
  % is cartesian.
\end{corlr}

\begin{proof}
  Note first that, as Zariski-open immersions are stable by pullbacks,
  both
  $\coprod_{i_{\ast}\gamma=\beta}\mathbb{R}\compmodsp[g]{n}(Z,\gamma)$
  and
  $\mathbb{R}\compmodsp[g]{n}(X,\beta) \mathop{\times}_{\mathbb{E}}
  \mathbb{R}\compmodsp[g]{n}(X,\beta)$ are open sub-derived stacks of
  $\rmap_{/\modstk[g]{n}}\left(\stkcurv[g]{n},Z\times\modstk[g]{n}\right)$,
  so by~\cite[Proposition 2.1]{schurg15:_deriv} to show that they are
  equal it is enough to show that their truncations define identical
  substacks of
  $\truncat\rmap_{/\modstk[g]{n}}\left(
    \stkcurv[g]{n},Z\times\modstk[g]{n}\right)$.

  As explained in~\cref{rmrk:truncat-rmap-thicken}, the truncation of
  such a derived mapping stack with (necessarily truncated) source
  flat over the truncated base is
  $\shname{Map}_{/\modstk[g]{n}}\left(\stkcurv[g]{n},
    \truncat(Z\times\modstk[g]{n})\right)$, and similarly
  \begin{equation}
    \truncat\left(\mathbb{R}\compmodsp[g]{n}(X,\beta)
      \mathop{\times}_{\mathbb{E}}
      \mathbb{R}\compmodsp[g]{n}(X,\beta)\right)
    =\compmodsp[g]{n}(X,\beta)
    \mathop{\times}_{\truncat\mathbb{E}}^{\func{t}}
    \compmodsp[g]{n}(X,\beta)\text{.} 
  \end{equation}
  In addition, the truncation $\infty$-functor commutes with colimits
  (see~\cite[proof of Lemma 2.2.4.1]{toen08:_homot}) so
  \begin{equation}
    \truncat\biggl(\coprod_{i_{\ast}\gamma=\beta}
      \mathbb{R}\compmodsp[g]{n}(Z,\gamma)\biggr)
    =\coprod_{i_{\ast}\gamma=\beta}\compmodsp[g]{n}(\truncat{Z},\gamma)\text{.}
  \end{equation}

  We now compare the two stacks (which are $1$-algebraic, and thus
  $1$-stacks, \emph{i.e.} taking values in $1$-groupoids)
  pointwise. For any $S\to\modstk[g]{n}$ (with corresponding prestable
  genus-$g$ curve $C_{S}\to S$), we have that
  $\bigl(\coprod_{i_{\ast}\gamma=\beta}
  \compmodsp[g]{n}(\truncat{Z},\gamma)\bigr)(S)
  =\coprod_{i_{\ast}\gamma=\beta}\compmodsp[g]{n}(\truncat{Z},\gamma)(S)$
  is tautologically the disjoint union (over
  $\gamma\in i_{\ast}^{-1}(\beta)$) of the groupoids of $S$-indexed
  families of stable maps from $C_{S}$ to $\truncat{Z}$ of class
  $\gamma$, and
  \begin{equation}
    \label{eq:22}
    \Bigl(\compmodsp[g]{n}(X,\beta)
    \mathop{\times}_{\truncat\mathbb{E}}^{\func{t}}
    \compmodsp[g]{n}(X,\beta)\Bigr)(S)
    \simeq\compmodsp[g]{n}(X,\beta)(S)
    \mathop{\times}_{\restr{\truncat\mathbb{E}}_{\compmodsp[g]{n}(X,\beta)}(S)}
    \compmodsp[g]{n}(X,\beta)(S)
  \end{equation}
  with
  $\restr{\truncat\mathbb{E}}_{\compmodsp[g]{n}(X,\beta)}(S)
  =\hom(C_{S},\truncat{E})$. The
  latter $(2,1)$-fibre product (of groupoids) consists of pairs of
  stable maps $f_{1},f_{2}$ from $C_{S}$ to $X$ equipped with an
  equivalence between their images $s\circ f_{1}$ and $0\circ f_{2}$
  in $\truncat{E}$,
  % is described in the
  % following way:
  % \begin{description}
  % \item[an object] consists of a pair of stable maps $f_{1},f_{2}$ from
  %   $C_{S}$ to $X$ and an isomorphism between their images in
  %   $\hom(C_{S},E)$, that is an automorphism $\varphi$ of $C_{S}$ such
  %   that $s\circ f_{1}=0\circ f_{2}\circ\varphi$, while
  % \item[an automorphism]
  %   $(f_{1},f_{2},\varphi)\simeq(f_{1},f_{2},\varphi)$ is given by a
  %   pair of automorphisms of $C_{S}$ compatible with all the data; or
  %   automorphisms of the two stable maps compatible with $\varphi$ (so
  %   that when $\varphi$ is not $\id_{C_{S}}$ the notion of
  %   automorphism is more rigid than the usual automorphisms of a
  %   single stable map).
  % \end{description}
  so that the obvious functor
  \begin{equation}
    \label{eq:23}
    \coprod_{i_{\ast}\gamma=\beta}\compmodsp[g]{n}(\truncat{Z},\gamma)(S)
    \to\Bigl(\compmodsp[g]{n}(X,\beta)
    \mathop{\times}_{\truncat\mathbb{E}}^{\func{t}}
    \compmodsp[g]{n}(X,\beta)\Bigr)(S)
  \end{equation}
  sending a stable map $f\colon C_{S}\to Z$ to
  $(i_{Z\hookrightarrow X}\circ f,i_{Z\hookrightarrow X}\circ
  f,\id_{C_{S}})$ is clearly an equivalence.
\end{proof}

We may now apply~\cref{pros:zerosec-struct-shf} to deduce a proof
of~\cref{thm:main-thm}. In the next~\cref{sec:funct-inter-theory}, we
will see how we can recover from this and~\cref{corlr:zero-sec-k-thry}
the classical (virtual) quantum Lefschetz formula.

\begin{remark}
  \label{rmrk:compar-cotgt-cplx-rmap}
  Further evidence for this geometric form of the quantum Lefschetz
  principle can also be found by comparing the tangent complexes. Let
  us write temporarily $M(X)$ and $M(Z)$ for the moduli stacks of
  stable maps $\mathbb{R}\compmodsp[g]{b}(X,\beta)$ and
  $\coprod_{i_{\ast}\gamma=\beta}\mathbb{R}\compmodsp[g]{n}(Z,\gamma)$,
  and $M^{\prime}(Z)$ for the zero locus
  $M(X)\times_{\mathbb{E}}M(X)$. The universal property of the latter
  stack induces a canonical morphism denoted
  $\Upsilon\colon M(Z)\to M^{\prime}(Z)$ such that
  $u_{i}^{\prime}\circ\Upsilon=u_{i}$ for $i=1,2$ where
  $u_{i}\colon M(Z)\hookrightarrow M(X)$ and
  $u_{i}^{\prime}\colon M^{\prime}(Z)\hookrightarrow M(X)$ are the
  canonical arrows (as in~\cref{eq:main-rslt-abstr-cart-diagr}).
  
  We know from~\cref{pros:map-stk-tgt-cplx} that
  $\mathbb{T}_{M(X)/\modstk[g]{n}}
  \simeq\restr{\func{p}_{g,n,\ast}\ev_{g,n}^{\ast}\mathbb{T}_{X}}
  _{M(X)}$. There is a fibre sequence
  $i_{1}^{\ast}\mathbb{L}_{X}\to\mathbb{L}_{Z}\to\mathbb{L}_{i_{1}\colon
    Z/X}$, and as $Z$ sit by definition in a cartesian square we have
  that
  $\mathbb{L}_{i_{1}}
  =i_{2}^{\ast}\mathbb{L}_{X/E}=\restr{\shname{E}^{\vee}[1]}_{Z}$
  (where once again we have written $i_{1,2}\colon Z\hookrightarrow X$
  the two canonical inclusions). As both pushforward and pullback
  preserve fibre sequences, we obtain finally that
  $\mathbb{T}_{M(Z)/\modstk[g]{n}}$ is the fibre of the morphism
  $\restr{\func{p}_{g,n,\ast}\ev_{g,n}^{\ast}\mathbb{T}_{X}}_{M(Z)}
  \to\restr{\func{p}_{g,n,\ast}\ev_{g,n}^{\ast}\shname{E}}_{M(Z)}$.

  Following the same logic, writing $M^{\prime}(Z)$ for the zero
  locus, we see that $\mathbb{T}_{M^{\prime}(Z)}$ is the fibre of
  $\mathbb{T}_{M(X)}
  =\restr{\func{p}_{g,n,\ast}\ev_{g,n}^{\ast}\mathbb{T}_{X}}_{M(X)}
  \to\restr{\func{p}_{g,n,\ast}\ev_{g,n}^{\ast}\shname{E}}_{M^{\prime}(Z)}$. But
  we have seen that
  $\Upsilon^{\ast}\circ u_{i}^{\prime,\ast}=u_{i}^{\ast}$ so it is
  clear that
  $\Upsilon^{\ast}\mathbb{T}_{M^{\prime}(Z)}\simeq\mathbb{T}_{M(Z)}$.
  
  As it is sufficient and necessary for a morphism of derived stacks
  to be an equivalence that it induce an isomorphism on the truncation
  and that its (co)tangent complex vanish, this is another way of
  proving~\cref{thm:identf-rmap-stk}.
\end{remark}

\begin{exmp}
  \label{exmp:lg-model-cotgt-vbndl}
  Let $(X,f\colon X\to\mathbb{A}^{1})$ be a Landau--Ginzburg model,
  from which we deduce the perfect cone
  $T^{\vee}X=\mathbb{V}_{X}(\mathbb{L}_{X})$ and section
  $d_{\mathrm{dR}}f$, whose zero locus is by definition the critical
  locus $\mathbb{R}\operatorname{Crit}(f)$ (which is the intersection
  of two Lagrangians in a $0$-shifted symplectic derived stack and
  thus carries a canonical $(-1)$-shifted symplectic form). Then the
  derived moduli stack of stable maps to
  $\mathbb{R}\operatorname{Crit}(f)$ is the zero locus of the induced
  section of
  \begin{equation}
    \label{eq:induced-vbnld-lg-model}
    \func{p}_{\ast}\ev^{\ast}T^{\vee}X
    =\mathbb{V}_{\mathbb{R}\compmodsp[g]{n}(X,\beta)}(
    \func{p}_{\ast}(\ev^{\ast}\mathbb{L}_{X}))
    \end{equation}
  But notice that
  \begin{equation}
    \label{eq:compar-cotgt-rmap-lg}
    T^{\vee}\mathbb{R}\compmodsp[g]{n}(X,\beta)
    \simeq\mathbb{V}_{\mathbb{R}\compmodsp[g]{n}(X,\beta)}(
    (\func{p}_{\ast}\ev^{\ast}\mathbb{T}_{X})^{\vee})
    \simeq\mathbb{V}_{\mathbb{R}\compmodsp[g]{n}(X,\beta)}(
    \func{p}_{!}\ev^{\ast}\mathbb{L}_{X})
  \end{equation}
  where
  $\func{p}_{!}\colon\shname{F}\mapsto\func{p}_{\ast}(\shname{F}^{\vee})^{\vee}
  \simeq\func{p}_{\ast}(\shname{F}\otimes\omega_{\func{p}})$ is the
  left adjoint to $\func{p}^{\ast}$ (by~\cite[Proposition
  6.4.5.3]{lurie19:_spect_algeb_geomet}), so
  $\mathbb{R}\compmodsp[g]{n}(\mathbb{R}\operatorname{Crit}(f),\beta)$
  is not a critical locus and so cannot in general be expected to
  carry a $(-1)$-shifted symplectic structure if $(g,n)$ differs from
  $(0,1)$ or $(1,0)$.

  It is also possible to go the other way, that is to obtain a
  Landau--Ginzburg model from our general setting. If
  $\varpi\colon E^{\vee}\to X$ is the dual of the perfect cone with
  section $s$, then the section $\varpi^{\ast}s$ of $\varpi^{\ast}E$
  can be paired with the tautological section $t$ of
  $\varpi^{\ast}E^{\vee}$, defining a function
  $w_{s}=\langle s,t\rangle$ on the total space
  $E^{\vee}$. By~\cite[Corollary
  3.8]{isik12:_equiv_deriv_categ_variet_singul_categ}, if $X$ is
  smooth, there is an equivalence
  $\cat{Coh}^{\mathrm{b}}(Z)
  \simeq\cat{Sing}(\mathbb{R}\operatorname{Zero}(w_{s})/\mathbb{G}_{\mathrm{m}})$
  with the $\mathbb{G}_{\mathrm{m}}$-equivariant dg-category of
  singularities of $\mathbb{R}\operatorname{Zero}(w_{s})$ (where
  $\mathbb{G}_{\mathrm{m}}$ acts by rescaling on the fibres of
  $E^{\vee}$). However we only have
  $Z=\mathbb{R}\operatorname{Crit}(w_{s})$ if $Z$ is smooth
  (see~\cite[Lemma 2.2.2]{chen19:_virtual} in the regular and
  underived case).
\end{exmp}

\section{Functoriality in intersection theory by the categorification
  of virtual pullbacks}
\label{sec:funct-inter-theory}

We have obtained (as~\cref{thm:main-thm}) a categorified form of the
quantum Lefschetz principle, which in the cases where
$\mathbb{E}=\func{p}_{\ast}\ev^{\ast}E$ is a vector bundle we can
by~\cref{corlr:zero-sec-k-thry} decategorify by passing to the
$G_{0}$-theory groups (or, more generally, the $G$-theory spectra) of
the derived moduli stacks. To show that our statement is indeed a
categorification of the quantum Lefschetz principle, it remains to
compare it with the virtual statement, in the $G$-theory of the
truncated moduli stacks. As explained in the introduction, this will
be obtained through an appropriate construction of virtual
pullbacks. These were defined in~\cite{manolache12:_virtual} (and
in~\cite{qu18:_virtual_k} for $G_{0}$-theory) from perfect obstruction
theories. Following the understanding of virtual classes and the
constructions of~\cite{mann18:_brane_gromov_witten_k}, we will give an
alternative construction from derived thickenings. To ensure
consistency, we show in~\cref{sec:comp-with-constr} that our
construction coincides with that of~\cite{qu18:_virtual_k} when both
are defined, and we use it in~\cref{sec:recov-quant-lefsch-kthry} to
get back the virtual form of the quantum Lefschetz formula.

\begin{remark}
  \label{rmrk:virt-pullbck-derived-history}
  The derived origin of virtual pullbacks was already considered
  in~\cite[Section
  7]{schurg11:_deriv_delig_mumfor_stack_perfec_obstr_theor}, where it
  is shown that any morphism of DM stacks which is the classical
  truncation of a morphism of derived DM stacks, with the induced
  obstruction theory, carries the compatibility necessary for the
  construction of a virtual pullback. However, the origin of the
  virtual classes and their precise relation to derived thickenings
  was still considered mysterious, and no direct construction of the
  virtual pullbacks from derived algebraic geometry was given.
\end{remark}

\subsection{Definition from derived geometry}
\label{sec:defin-from-deriv}

Let $f\colon X\to Y$ be a quasi-smooth morphism of derived stacks,
that is its cotangent complex $\mathbb{L}_{f\colon X/Y}$ is of perfect
Tor-amplitude in $[-1,0]$.

\begin{remark}
  \label{rmrk:qsm-pullbck-preserv-coh-g-thry}
  By \cite[Chapter 4, Lemma 3.1.3]{gaitsgory17}, as the quasi-smooth
  morphism $f$ is of finite Tor-amplitude, the pullback of
  quasicoherent sheaves $f^{\ast}$ maps $\cat{Coh}^{\mathrm{b}}(Y)$
  to $\cat{Coh}^{\mathrm{b}}(X)$. As we work in $G$-theory, which is
  the $K$-theory of the stable $\infty$-category of bounded coherent
  sheaves, the notation $f^{\ast}$ will be understood in
  this~\namecref{sec:funct-inter-theory} to mean the restriction of
  the pullback operation to coherent sheaves.
\end{remark}

Recall that, due to the theorem of the heart (\emph{cf.}~\cite[Theorem
6.1]{barwick15:_theor_heart}) and~\cite[Corollary 2.5.9.2 with
$n=0$]{lurie19:_spect_algeb_geomet}, the closed embedding
$\jmath_{X}\colon\truncat{X}\hookrightarrow X$ induces an
equivalence
$\jmath_{X,\ast}\colon G(\truncat{X})\xrightarrow{\simeq}G(X)$ in
$G$-theory, whose inverse at the level of $G_{0}$-groups is given by
$\bigl(\pi_{0}(\jmath_{X,\ast})\bigr)^{-1}\colon G_{0}(X)\ni\shname{G}
\mapsto\sum_{i\geq0}(-1)^{i}[\pi_{i}(\shname{G})]\in
G_{0}(\truncat{X})$.

It is therefore natural to define the virtual pullback along
$\truncat{f}$ to be given by the actual pullback along $f$,
intertwined with these isomorphisms.

However we wish to consider the virtual pullback as a bivariant class,
that is defined as a collection of maps
$G(Y^{\prime})\to
G(X\times_{Y}Y^{\prime})=G(X\times_{Y}^{\func{t}}Y^{\prime})$ indexed
by all $\truncat{Y}$-schemes $Y^{\prime}\to\truncat{Y}$, or more
generally by all derived $Y$-schemes $Y^{\prime}\to Y$. Then the
virtual pullback we defined should be the map corresponding to the
$\truncat{Y}$-scheme $\id_{\truncat{Y}}\colon\truncat{Y}=\truncat{Y}$.

We recall that we use the notation $\times^{\func{t}}$ (a fibre
product decorated by $\func{t}$) to differentiate the ``truncated''
($1$- $2$-categorical, for our moduli $1$-stacks) fibre products of
classical stacks from the implictly $\infty$-categorical fibre
products of derived stacks.

\begin{definition}
  \label{def:virt-pullback-bivar}
  The \textbf{bivariant virtual pullback} along $f$ is the collection,
  indexed by all $Y$-schemes $a\colon Y^{\prime}\to Y$, of maps
  $(\truncat{f})^{!,a}_{\mathrm{DAG}}\colon G(\truncat{Y}^{\prime})\to
  G(\truncat(Y^{\prime}\times_{Y}X))$ defined as follows.

  For a morphism of schemes $a\colon Y^{\prime}\to Y$, we have the
  diagram
  \begin{equation}
    \label{eq:def-vpullback-base}
    \begin{tikzcd}
      \truncat\Bigl(Y^{\prime}\lmtimes_{Y}X\Bigr)\simeq
      \truncat{Y}^{\prime}\lmtimes_{\truncat{Y}}^{\func{t}}\truncat{X}
      \arrow[dr,bend right]
      \arrow[r,hook,"\jmath_{Y^{\prime}\times_{Y}X}"] & Y^{\prime}\lmtimes_{Y}X
      \arrow[d] \arrow[r,"\widetilde{f}"] \arrow[dr,phantom,very near
      start,"\lrcorner"] & Y^{\prime}
      \arrow[d,"a"] \\
      & X \arrow[r,"f"'] & Y \\
    \end{tikzcd}\text{.}
  \end{equation}%
  Then we set
  $(\truncat{f})^{!,a}_{\mathrm{DAG}}
  \coloneqq(\jmath_{Y^{\prime}\times_{Y}X,\ast})^{-1}
  \circ\widetilde{f}^{\ast}\circ\jmath_{Y^{\prime},\ast}$.
\end{definition}

\begin{lemma}
  \label{lem:vpullbck-indep-derived}
  The virtual pullback only depends on
  $\truncat{a}\colon\truncat{Y}^{\prime}\to\truncat{Y}$. That is, for
  any $a_{1},a_{2}\colon Y^{\prime}_{1},Y^{\prime}_{2}\to Y$ with
  $\truncat{a_{1}}=\truncat{a_{2}}$, the virtual pullbacks
  $f^{!,a_{1}}_{\mathrm{DAG}}$ and $f^{!,a_{2}}_{\mathrm{DAG}}$
  induced by $a_{1}$ and $a_{2}$ are equivalent.
\end{lemma}

\begin{proof}
  For any $a\colon Y^{\prime}\to Y$, we compare the virtual pullbacks
  induced by $a$ and
  $\truncat{Y}^{\prime}\xrightarrow{\truncat{a}}
  \truncat{Y}\xrightarrow{\jmath_{Y}}Y$.
  \begin{equation}
    \label{eq:vpllbck-trunc-base-diagr}
    \begin{tikzcd}[cramped]
      \truncat{Y}^{\prime}\lmtimes_{\truncat{Y}}^{\func{t}}\truncat{X}
      \arrow[rr,bend left=5,hook] \arrow[dr,bend right,hook] & &
      Y^{\prime}\times_{Y}X
      \arrow[rr,"\widetilde{f}"] & & Y^{\prime} \arrow[dd,"a"] \\
      & \truncat(Y^{\prime})\times_{Y}X \arrow[rr,near
      start,"\widehat{f}"'] \arrow[dr] \arrow[ur,"i"]
      \arrow[urrr,phantom,very near start,"\urcorner"] & &
      \truncat{Y^{\prime}} \arrow[dr,near
      start,"\jmath_{Y}\circ\truncat{a}"']
      \arrow[ur,hook] & \\
      & & X \arrow[rr,"f"'] \arrow[from=uu,crossing over] & & Y
    \end{tikzcd}
  \end{equation}
  The back square (the one exhibiting
  $\widetilde{f}\circ i\simeq \jmath_{Y^{\prime}}\circ\widehat{f}$) is
  cartesian and its side $\jmath_{Y^{\prime}}$ is a closed immersion
  and thus proper, so the base-change formula gives
  $\widetilde{f}^{\ast}\circ\jmath_{Y^{\prime},\ast}=i_{\ast}\widehat{f}^{\ast}$.
  Commutativity of the leftmost triangle implies that
  $i_{\ast}\jmath_{\truncat{Y}^{\prime}\times_{Y}X,\ast}
  =\jmath_{Y^{\prime}\times_{Y}X,\ast}$,
  and as both closed immersions involved induce isomorphisms in
  $G$-theory, we have
  $(\jmath_{Y^{\prime}\times_{Y}X,\ast})^{-1}i_{\ast}
  =(\jmath_{\truncat{Y^{\prime}}\times_{Y}X,\ast})^{-1}$. Putting the
  ingredients together, we finally obtain that
  \begin{equation}
    \label{eq:vpllbck-trunc-isom}
    f^{!,a}_{\mathrm{DAG}}
    \coloneqq(\jmath_{Y^{\prime}\times_{Y}X,\ast})^{-1}
    \widetilde{f}^{\ast}\jmath_{Y^{\prime},\ast}
    =(\jmath_{Y^{\prime}\times_{Y}X,\ast})^{-1}i_{\ast}\widehat{f}^{\ast}
    =(\jmath_{\truncat{Y^{\prime}}\times_{Y}X,\ast})^{-1}\widehat{f}^{\ast}
    \eqqcolon f^{!,\jmath_{Y}\circ\truncat{a}}_{\mathrm{DAG}}\text{.}
  \end{equation}
\end{proof}

\begin{remark}[Functoriality]\label{rmrk:vpllbck-fctorl}
  The virtual pullbacks satisfy obvious functoriality properties. Let
  $X\xrightarrow{f}Y\xrightarrow{g}Z$ be two composable arrows, and
  let $a\colon Z^{\prime}\to Z$ be a $Z$-scheme. We have the
  commutative diagram
  \begin{equation}
    \label{eq:vpllbck-fctrl-diagr}
    \begin{tikzcd}
      \truncat(Z^{\prime}\times_{Z}X)
      \arrow[dr,hook,"\jmath_{Z^{\prime}\times_{Z}X}"'] &
      \truncat(Z^{\prime}\times_{Z}X)
      \arrow[dr,hook,"\jmath_{Z^{\prime}\times_{Z}Y}"] & & \\
      & Z^{\prime}\times_{Z}X \arrow[d] \arrow[r,"\widetilde{f}"] &
      Z^{\prime}\times_{Z}Y \arrow[r,"\widetilde{g}"]
      \arrow[d,"b"] & Z^{\prime} \arrow[d,"a"] \\
      & X \arrow[r,"f"'] & Y \arrow[r,"g"'] & Z
    \end{tikzcd}
  \end{equation}

  It follows by associativity of fibre products that
  \begin{equation}
    \label{eq:vpllbck-fctrl-compat-comput}
    \begin{split}
      (\truncat{f})^{!,b}_{\mathrm{DAG}}\circ(\truncat{g})^{!,a}_{\mathrm{DAG}}
      &=(\jmath_{(Z^{\prime}\times_{Z}Y)\times_{Y}X,\ast})^{-1}
      \circ\widetilde{f}^{\ast}\circ(\jmath_{Z^{\prime}\times_{Z}Y,\ast})
      \circ(\jmath_{Z^{\prime}\times_{Z}Y,\ast})^{-1}
      \circ\widetilde{g}^{\ast}\circ\jmath_{Z^{\prime},\ast}\\
      &=(\jmath_{Z^{\prime}\times_{Z}X,\ast})^{-1}
      \circ\widetilde{gf}^{\ast}\circ\jmath_{Z^{\prime},\ast}
      \eqqcolon(\truncat(gf))^{!,a}_{\mathrm{DAG}}\text{.}
    \end{split}
  \end{equation}
\end{remark}

\subsection{Comparison with the construction from obstruction
  theories}
\label{sec:comp-with-constr}

\begin{construct}[Virtual pullbacks from perfect obstruction
  theories]
  \label{constr:vpllbck-thickn-obstr-thry}
  Let $g\colon V\to W$ be a morphism of Artin stacks of
  Deligne--Mumford type (\emph{i.e.} relatively DM) endowed with a
  perfect obstruction theory
  $\varphi\colon E\to\mathbb{L}_{g\colon V/W}$, inducing the closed
  immersion
  $\varphi^{\vee}\colon\mathfrak{C}_{g\colon
    V/W}\hookrightarrow\mathfrak{E}$, where
  $\mathfrak{E}=\truncat(\mathbb{V}_{V}(E[1]^{\vee}))$ is the
  vector bundle (Picard) stack associated with $E$ and
  $\mathfrak{C}_{g}$ is the intrinsic normal cone of $g$ (constructed
  in~\cite{behrend97}). As in~\cite{mann18:_brane_gromov_witten_k} we
  define a derived thickening $\mathbb{R}^{\varphi}V$ of $V$ as the
  derived intersection
  \begin{equation}
    \label{eq:vpllbck-thickn-obstr-thry-diagr}
    \begin{tikzcd}
      \mathbb{R}^{\varphi}V \arrow[dr,phantom,very near
      start,"\lrcorner"] \arrow[d,"p"'] \arrow[r,"q"] &
      \mathfrak{C}_{g} \arrow[d,hook,"\varphi^{\vee}"] \\
      V \arrow[r,hook,"0_{\mathfrak{E}}"'] & \mathfrak{E}
    \end{tikzcd}\text{.}
  \end{equation}
  Note that the arrow $p$ is a retract of $\jmath_{V}$, and provides a
  splitting of the induced perfect obstruction theory
  $\jmath_{V}^{\ast}\mathbb{L}_{\mathbb{R}^{\varphi}V}\to\mathbb{L}_{V}$. We
  may use it to define a map of derived stacks
  $\mathbb{R}^{\varphi}g\colon
  \mathbb{R}^{\varphi}V\xrightarrow{p}V\xrightarrow{f}W$ which is a
  derived thickening of $g$.
\end{construct}

We also recall the construction of the virtual pullback
$g^{!}_{\varphi}$, or $g^{!}_{\mathrm{POT}}$, from the perfect
obstruction $\varphi$, defined in~\cite{manolache12:_virtual} for Chow
homology then~\cite{qu18:_virtual_k} for $G_{0}$-theory.

Let $a\colon W^{\prime}\to W$ and write
$g^{\prime}\colon V^{\prime}\to W^{\prime}$ the base-change of $g$.
Recall that one may define a deformation space (constructed
in~\cite[Theorem 4.1.13]{khan19:_virtual_cartier} for quasi-smooth
closed immersions of derived stacks, and extended by~\cite[Proposition
7.6.2]{hekking21:_graded} to arbitrary closed immersions)
$\mathfrak{D}_{V^{\prime}}W^{\prime}$ over $\mathbb{P}^{1}_{\Bbbk}$,
with general fibre $W^{\prime}$ giving the open immersion
$\mathfrak{j}\colon W^{\prime}\times\mathbb{A}^{1}_{\Bbbk}
\hookrightarrow\mathfrak{D}_{V^{\prime}}W^{\prime}$, and special fibre
$\mathfrak{C}_{g^{\prime}}$ giving the complementary closed immersion
$\mathfrak{i}\colon\mathfrak{C}_{g^{\prime}}\times\{\infty\}
\hookrightarrow\mathfrak{D}_{V^{\prime}}W^{\prime}$. It follows that
there is an exact sequence of abelian groups
$G_{0}(\mathfrak{C}_{g^{\prime}})\to
G_{0}(\mathfrak{D}_{V^{\prime}}W^{\prime})\to
G_{0}(W^{\prime}\times\mathbb{A}^{1})\to0$ (coming from the fibred
sequence of $G$-theory spectra). Furthermore, as (by excess
intersection) $\mathfrak{i}^{\ast}\mathfrak{i}_{\ast}$ is equivalent
to tensoring by the symmetric algebra on the conormal bundle of
$\mathfrak{C}_{g^{\prime}}$ in $\mathfrak{D}_{V^{\prime}}W^{\prime}$
and as the latter is trivial, we have
$\mathfrak{i}^{\ast}\mathfrak{i}_{\ast}=0$, inducing a map
$G_{0}(W^{\prime}\times\mathbb{A}^{1})\to
G_{0}(\mathfrak{C}_{g^{\prime}})$: concretely, any section
$\mathfrak{j}^{\ast,-1}$ of $\mathfrak{j}^{\ast}$ gives the same map
when post-composed with $\mathfrak{i}^{\ast}$ so we do have a
well-defined map $\mathfrak{i}^{\ast}\mathfrak{j}^{\ast,-1}$. The
specialisation map
$\specialis\colon G_{0}(W^{\prime})\to
G_{0}(\mathfrak{C}_{g^{\prime}})$ is then defined by precomposing it
by
$\pr^{\ast}\colon G_{0}(W^{\prime})\to
G_{0}(W^{\prime}\times\mathbb{A}^{1})$. Finally, the cartesian square
defining $V^{\prime}$ induces by~\cite[Proposition
2.26]{manolache12:_virtual} a closed immersion
$c\colon\mathfrak{C}_{g^{\prime}}\hookrightarrow
a^{\ast}\mathfrak{C}_{g}=V^{\prime}\times_{V}\mathfrak{C}_{g}$, and
the virtual pullback $g^{!,a}_{\varphi}$ along $g$ is constructed as
the composite
\begin{equation}
  \label{eq:vpullback-pot-def}
  g^{!,a}_{\varphi}\colon G_{0}(W^{\prime})
  \xrightarrow{\specialis}G_{0}(\mathfrak{C}_{g^{\prime}})
  \xrightarrow{c_{\ast}}
  G_{0}(a^{\ast}\mathfrak{C}_{g})
  \xrightarrow{(a^{\ast}\varphi^{\vee})_{\ast}}
  G_{0}(a^{\ast}\mathfrak{E}=V^{\prime}\times_{V}\mathfrak{E})
  \xrightarrow{0_{a^{\ast}\mathfrak{E}}^{\ast}}G_{0}(V^{\prime})\text{.}
\end{equation}

\begin{lemma}
  \label{lem:compar-vpllbck-obstr-thry}
  The virtual pullback
  $(\truncat\mathbb{R}^{\varphi}g)^{!}_{\mathrm{DAG}}$ as defined
  above for the map $\mathbb{R}^{\varphi}g$ coincides with the virtual
  pullback $g^{!}_{\varphi}$
  of~\cite{manolache12:_virtual,qu18:_virtual_k}: for any
  $a\colon W^{\prime}\to W$, we have
  $(\truncat\mathbb{R}^{\varphi}g)^{!,a}_{\mathrm{DAG}}
  =g^{!,a}_{\varphi}\colon G_{0}(W^{\prime})\to G_{0}(V^{\prime})$.
\end{lemma}

Note that this is essentially also proved as~\cite[Proposition
6.8]{khan22:_k_g}.

\begin{proof}
  We adapt the results of~\cite[Proposition 3.5]{joshua10:_rieman} to
  the more general case of a morphism that need not be a regular
  embedding.
  
  Let again $a\colon W^{\prime}\to W$ and write
  $g^{\prime}\colon V^{\prime}\to W^{\prime}$ the base-change of $g$.
  We now review our construction of the virtual pullback from derived
  thickenings from the point of view of the perfect obstruction
  theory. The map $g^{!,a}_{\mathrm{DAG}}$
  of~\cref{def:virt-pullback-bivar} is computed in the following way:
  we define a derived thickening $\mathbb{R}^{\varphi}V^{\prime}$ of
  $V^{\prime}=V\times_{W}^{\func{t}}W^{\prime}$ as
  $V^{\prime}\times_{\mathfrak{E}}\mathfrak{C}_{g}$; note that we have
  $\mathbb{R}^{\varphi}V^{\prime}=V^{\prime}\times_{V}\mathbb{R}^{\varphi}V$
  and writing
  $p^{\prime}\colon\mathbb{R}^{\varphi}V^{\prime}\to V^{\prime}$ we
  obtain a derived thickening
  $\mathbb{R}^{\varphi}g^{\prime}=g^{\prime}\circ p^{\prime}
  \colon\mathbb{R}^{\varphi}V^{\prime}\to W^{\prime}$ of
  $g^{\prime}$. Then $g^{!,a}_{\mathrm{DAG}}$ is the pullback along
  $\mathbb{R}^{\varphi}g^{\prime}$ followed by the inverse of
  $\jmath_{\mathbb{R}^{\varphi}V^{\prime},\ast}$.
  
  We also note that the fibred product
  $V^{\prime}\times_{a^{\ast}\mathfrak{E}}(a^{\ast}\mathfrak{C}_{g})$
  is the base-change of $V\times_{\mathfrak{E}}\mathfrak{C}_{g}$ along
  $a^{\prime}\colon V^{\prime}\to V$, so the square
  \begin{equation}
    \label{eq:vpllbck-thickn-obstr-thry-diagr-basechgd}
    \begin{tikzcd}
      \mathbb{R}^{\varphi}V^{\prime} \arrow[dr,phantom,very near
      start,"\lrcorner"] \arrow[d,"p^{\prime}"']
      \arrow[r,"q^{\prime}"] &
      a^{\ast}\mathfrak{C}_{g} \arrow[d,hook,"a^{\ast}\varphi^{\vee}"] \\
      V^{\prime} \arrow[r,hook,"0_{a^{\ast}E}"'] &
      a^{\ast}\mathfrak{E}
    \end{tikzcd}
  \end{equation}
  is cartesian. As $p^{\prime}$ is proper, we have
  $0_{a^{\ast}\mathfrak{E}}^{\ast}(a^{\ast}\varphi^{\vee})_{\ast}
  =p^{\prime}_{\ast}q^{\prime,\ast}$; concomitantly, as $p^{\prime}$
  is a retract of $\jmath_{\mathbb{R}^{\varphi}V^{\prime}}$ we have in
  $G$-theory
  $(\jmath_{\mathbb{R}^{\varphi}V^{\prime},\ast})^{-1}=p^{\prime}_{\ast}$. We
  conclude that the virtual pullback of~\cite{qu18:_virtual_k}
  coincides with
  $(\jmath_{\mathbb{R}^{\varphi}V^{\prime},\ast})^{-1}\circ
  q^{\prime,\ast}\circ c_{\ast}\circ\specialis$, and thus it only
  remains to check that the latter part specialises to
  $(\mathbb{R}^{\varphi}g^{\prime})^{\ast} =p^{\prime,\ast}\circ
  g^{\prime,\ast}$. But the deformation space
  $\mathfrak{D}_{V^{\prime}}W^{\prime}$ provides exactly an
  interpolation between $g^{\prime}\colon V^{\prime}\to W^{\prime}$
  and $V^{\prime}\hookrightarrow\mathfrak{C}_{g^{\prime}}$, so by
  transporting this comparison along the $\mathbb{A}^{1}$-invariance
  of $G$-theory the~\namecref{lem:compar-vpllbck-obstr-thry} is
  proved.
\end{proof}

Recall that for any quasi-smooth morphism $f\colon X\to Y$ of derived
Artin stacks, by~\cite[Proposition 1.2]{schurg15:_deriv} the canonical
map
$\varphi\colon
j_{X}^{\ast}\mathbb{L}_{f}\to\mathbb{L}_{\truncat{f}}$ is a perfect
obstruction theory.

\begin{pros}
  \label{pros:compar-vpllbck-derivd-pot}
  Let $f\colon X\to Y$ be a quasi-smooth relatively DM map of derived
  Artin stacks. The virtual pullback
  $(\truncat{f})^{!}_{\mathrm{DAG}}$ defined with derived geometry
  is equal to
  $(\truncat\mathbb{R}^{\varphi}\truncat{f})^{!}_{\mathrm{DAG}}$,
  and thus to the virtual pullback $(\truncat{f})^{!}_{\varphi}$
  of~\cite{manolache12:_virtual,qu18:_virtual_k}, induced by the
  obstruction theory
  $\varphi\colon
  j_{X}^{\ast}\mathbb{L}_{f}\to\mathbb{L}_{\truncat{f}}$.
\end{pros}

\begin{proof}
  The proof is similar to the one sketched in~\cite[Proposition
  4.3.2]{mann18:_brane_gromov_witten_k} for the comparison of the
  virtual classes defined from perfect obstruction theories and
  derived geometry, which mainly
  followed~\cite{lowrey12:_groth_rieman_roch}: one constructs a
  deformation to the normal bundle of the closed immersion
  $\jmath_{X}\colon\truncat{X}\hookrightarrow X$, and finally uses
  that $G$-theory is $\mathbb{A}^{1}$-invariant. Note that the main
  ingredient which was missing to make the proof
  of~\cite{mann18:_brane_gromov_witten_k} precise, deformation to the
  normal cone for derived stacks, has now been constructed
  by~\cite{hekking21:_graded}.
\end{proof}

We shall henceforth simply write $(\truncat{f})^{!}$ for the virtual
pullback along $f$.

\begin{exmp}[Virtual classes]
  Suppose $Y=\spec(\Bbbk)$ so $f\colon X\to\spec(\Bbbk)$ is
  the structure morphism. The virtual structure sheaf of
  $\truncat{X}$ is
  $\virshf{\truncat{X}}=f^{!,\id_{X}}([\shname{O}_{\spec(\Bbbk)}])
  =(\jmath_{X,\ast})^{-1}([\shname{O}_{X}])$.
\end{exmp}

\begin{exmp}
  Suppose that the classical map $g$ is already a quasi-smooth
  immersion, so that $\id_{\mathbb{L}_{g}}$ is a perfect obstruction
  theory. Then the virtual pullback is given by the Gysin pullback
  $g^{!}$, studied in details for example in~\cite{joshua10:_rieman}.
\end{exmp}

\begin{remark}[Virtual pullbacks in generalised motivic homology
  theories]
  \label{rmrk:genal-hmlgy-thries}
  Our construction of virtual pullbacks only relies on the fact that
  $G$-theory is insensitive to the non-reduced structure, and the
  identification with the classical definition requires simply the
  specialisation morphism and, more generally, the
  $\mathbb{A}^{1}$-invariance. These ingredients are present in
  motivic homotopy theory (by construction for the
  $\mathbb{A}^{1}$-invariance, and by~\cite[Corollary
  3.2.9]{khan19:_morel_voevod} for the insentivity to derived
  structures), so the virtual pullbacks in motivic cohomology theories
  also admit the derived geometric interpretation.

  In fact such virtual pullbacks were constructed for motivic
  Borel--Moore homology with coefficients in any étale motivic
  spectrum in~\cite[Construction 3.4]{khan19:_virtual_i} from the
  virtual pullbacks canonically associated with a quasi-smooth derived
  enhancement (through its derived deformation space).
\end{remark}

\subsection{Recovering the quantum Lefschetz formula}
\label{sec:recov-quant-lefsch-kthry}

\begin{pros}
  With the notations of~\cref{sec:excess-inters-form}, if $\shname{F}$
  is a vector bundle then
  $(\truncat{u})_{\ast}\virshf{T}
  =\virshf{M}\otimes\lambda_{-1}(\pi_{0}\shname{F}^{\vee})$ in $G(M)$.
\end{pros}

\begin{proof}
  By naturality of the transformation $\jmath$, we have
  $(\truncat{u})_{\ast}=(\jmath_{M,\ast})^{-1}u_{\ast}\jmath_{T,\ast}$
  so that
  $(\truncat{u})_{\ast}(\truncat{u})^{!}
  =(\jmath_{M,\ast})^{-1}u_{\ast}u^{\ast}\jmath_{M,\ast}
  =(\jmath_{M,\ast})^{-1}(\jmath_{M,\ast}(-)\otimes\lambda_{-1}(\shname{F}^{\vee}))$
  by~\cref{corlr:zero-sec-k-thry}. Hence
  $(\truncat{u})_{\ast}\virshf{T}
  =(\truncat{u})_{\ast}(\truncat{u})^{!}\virshf{M}
  =(\jmath_{M,\ast})^{-1}(\lambda_{-1}(\shname{F}^{\vee}))$.

  By~\cite[Corollary 25.2.3.3]{lurie19:_spect_algeb_geomet}, as
  $\shname{F}^{\vee}$ is flat over $\shname{O}_{M}$ so are its
  exterior powers $\bigwedge^{n}(\shname{F}^{\vee})$. In particular,
  by~\cite[Proposition 2.2.2.5. (4)]{toen08:_homot} they are strong
  $\shname{O}_{M}$-modules, meaning that
  $\pi_{i}(\bigwedge^{n}\shname{F}^{\vee})
  \simeq\pi_{i}(\shname{O}_{M})\otimes_{\pi_{0}(\shname{O}_{M})}
  \pi_{0}(\bigwedge^{n}\shname{F}^{\vee})$ for all natural integers
  $i$ and $n$, and we conclude that
  \begin{equation}
    \label{eq:comput-virshf-strong}
    \begin{split}
      (\truncat{u})_{\ast}\virshf{T}
      &=\sum_{i\geq0}(-1)^{i}\sum_{n\geq0}(-1)^{n}
      \left[\pi_{i}\left(\bigwedge\nolimits^{n}\shname{F}^{\vee}\right)\right]\\
      &=\sum_{i\geq0}(-1)^{i}[\pi_{i}(\shname{O}_{M})]\otimes
      \sum_{n\geq0}(-1)^{n}\left[
        \bigwedge\nolimits^{n}\pi_{0}(\shname{F}^{\vee})\right]
  \end{split}
  \end{equation}
  as required.
\end{proof}

\begin{remark}
  \label{remark:conditions-convex-necess}
  In the setting of the quantum Lefschetz principle, the only cases in
  which $\mathbb{E}_{g,n}$ is a vector bundle are when $E$ is convex,
  that is $\mathbb{R}^{1}p_{\ast}f^{\ast}\shname{E}=0$ for any stable
  map $(p\colon C\to S,f\colon C\to X)$ from a rational curve $C$, and
  thus the genus is $g=0$, which is the setting in which the quantum
  Lefschetz principle is already known. We conclude that it is not
  possible to relax the hypotheses for the quantum Lefschetz principle
  in $G$-theory, and that the more general version is thus only valid
  in its categorified form.

  One may also notice that as the cotangent complex of $u$ is
  $\func{p}_{\ast}\ev^{\ast}\shname{E}^{\vee}[1]$, which has
  Tor-amplitude in $[-2,0]$ (in fact $[-2,-1]$) unless the above
  conditions are satisfied, so that $u$ is not quasi-smooth and the
  virtual pullback along it cannot be defined.
\end{remark}

\begin{corlr}
  \label{corlr:get-back-qlefshtz-gthry}
  If $\mathbb{E}_{0,n}=\func{p}_{0,n;\ast}\ev_{0,n}^{\ast}E$ is a
  vector bundle (that is if $E$ is convex), the $G$-theoretic quantum
  Lefschetz formula of~\cref{thm:state-art-qlefschetz} holds:
  \begin{equation}
    \label{eq:qlefsch-form-gthry-final}
    (\truncat{u})_{\ast}\sum_{i_{\ast}\gamma=\beta}\virshf{\compmodspo(Z,\gamma)}
    =\virshf{\compmodspo(X,\beta)}\otimes
    \lambda_{-1}(\pi_{0}\func{p}_{0,n;\ast}\ev_{0,n}^{\ast}\shname{E}^{\vee})\text{.}
  \end{equation}\qed{}
\end{corlr}

% \bibliographystyle{halpha}
% \bibliography{Biblio}
\printbibliography

\end{document}